\documentclass[final]{amsart}
\usepackage{amsthm}
\usepackage{amsmath}
\usepackage{amsfonts}
\usepackage{amssymb}
\usepackage{showkeys}
\usepackage[usenames]{color}
\usepackage[mathscr]{euscript}
\usepackage[hidelinks]{hyperref}
\usepackage[arrow,curve,frame,color]{xy}
\def\natural{{\mathbb N}}
\def\zahlen{{\mathbb Z}}
\def\real{{\mathbb{R}}}
\def\biglip{\text{\normalfont Lip}} 
\def\Xint#1{\mathchoice
   {\XXint\displaystyle\textstyle{#1}}%
   {\XXint\textstyle\scriptstyle{#1}}%
   {\XXint\scriptstyle\scriptscriptstyle{#1}}%
   {\XXint\scriptscriptstyle\scriptscriptstyle{#1}}%
   \!\int}
\def\XXint#1#2#3{{\setbox0=\hbox{$#1{#2#3}{\int}$}
     \vcenter{\hbox{$#2#3$}}\kern-.5\wd0}}

\def\av{\Xint-}
\def\locnorm#1,#2.{\left|{#1}\right|_{{#2},\text{\normalfont loc}}}
\DeclareMathOperator\weight{weight}
\def\lebmeas#1.{\setbox1=\hbox{$#1$\unskip}{\mathcal L}^{\ifdim\wd1>0pt
    #1 \else 1 \fi}} 
\def\hmeas#1.{\mathscr{H}^{\setbox0=\hbox{$#1\unskip$}\ifdim\wd0=0pt 1
    \else #1\fi}} 
\newcommand{\on}{\:\mbox{\rule{0.1ex}{1.2ex}\rule{1.1ex}{0.1ex}}\:}
\def\cmass#1.{{\|#1\|}}
\def\wder#1.{{\mathscr{X}}({#1})}
\def\wform#1.{{\mathscr{E}}({#1})}
\DeclareMathOperator\spt{spt} 
\DeclareMathOperator\im{im}
\def\munsplit{8mu} 
\def\hmeas#1.{\mathscr{H}^{\setbox0=\hbox{$#1\unskip$}\ifdim\wd0=0pt 1
    \else #1\fi}} 
\def\pcreatenrm#1#2#3{\expandafter\def\csname
  #1nrm\endcsname##1.{\left #3 ##1\right #3_{#2}} \expandafter\def\csname
  #1nrmname\endcsname{\left #3\,\cdot\,\right #3_{#2}}} 
\def\pcreatecurrentnrm#1#2{\expandafter\def\csname
  #1nrm\endcsname##1.{{#2}\left ( ##1\right )} \expandafter\def\csname
  #1nrmname\endcsname{{#2}\left (\,\cdot\,\right )}} 
\def\pcreateasymnrm#1#2#3#4{\expandafter\def\csname
  #1nrm\endcsname##1.{\left #3 ##1\right #4_{#2}} \expandafter\def\csname
  #1nrmname\endcsname{\left #3\,\cdot\,\right #4_{#2}}} 
\newcommand{\dist}{\operatorname{dist}}
\def\lipalgb#1,#2.{{\rm Lip}_{\text{\normalfont b},#2}(#1)}
\def\lipalg#1.{{\rm Lip}_{\text{\normalfont b}}(#1)} 
\DeclareMathOperator\diam{diam}
\def\glip#1.{{\bf L}(#1)} 
\DeclareMathOperator\curves{Curves}
\def\albrep#1.{{\mathcal A}_{#1}} 
\def\hwder#1,#2.{\setbox1=\hbox{$#2$\unskip}\mathscr{X}^{\ifdim\wd1>0pt
  #2\else k\fi}({#1})}
\def\hwform#1,#2.{\setbox1=\hbox{$#2$\unskip}\mathscr{E}^{\ifdim\wd1>0pt
  #2\else k\fi}({#1})} 
\def\adsys{\{(X_i,\mu_i)\}_{i\in I}}

\def\brad{\text{\normalfont rad}} 
\def\ccell#1,#2.{\setbox1=\hbox{$#1$\unskip}\setbox2=\hbox{$#2$\unskip} 
  		 \text{\normalfont Cell}_{\ifdim\wd2>0pt #2\else N\fi}
		 ({\ifdim\wd1>0pt #1\else X_i^{(1)}\fi})}
\def\bdcell#1,#2.{\setbox1=\hbox{$#1$\unskip}\setbox2=\hbox{$#2$\unskip} 
                 \text{\normalfont Bd}({\ifdim\wd1>0pt #1\else f_i\fi}
                 {\ifdim\wd2>0pt ,#2\fi})}
\pcreatecurrentnrm{flat}{\text{\normalfont Flat}}
\pcreatecurrentnrm{nc}{\text{\normalfont\textbf{N}}}
\pcreatecurrentnrm{mc}{\text{\normalfont\textbf{M}}}
\def\makeunderscoreletter{\catcode`_=11}
\def\makecolonletter{\catcode`:=11}
\def\unmakecolonletter{\catcode`:=12}
\def\makeunderscoresub{\catcode`_=8}
\def\sync #1.{\makeunderscoreletter\makecolonletter\expandafter\ifx\csname
  sync#1\endcsname\relax Undefined\else \csname
  sync#1\endcsname\fi\makeunderscoresub\unmakecolonletter}
\makeunderscoreletter\makecolonletter\def\syncthm:normalprecurr{
5.35}\makeunderscoresub\unmakecolonletter
\makeunderscoreletter\makecolonletter\def\synclem:locality{
27}\makeunderscoresub\unmakecolonletter
\makeunderscoreletter\makecolonletter\def\syncdef:ext{ 
7.9}\makeunderscoresub\unmakecolonletter
\makeunderscoreletter\makecolonletter\def\syncrem:ext{
5.1}\makeunderscoresub\unmakecolonletter
\makeunderscoreletter\makecolonletter\def\syncthm:repr{
1.4}\makeunderscoresub\unmakecolonletter
\numberwithin{equation}{section} 
\theoremstyle{plain}

\newtheorem{thm}[equation]{Theorem}

\theoremstyle{definition}
\newtheorem{defn}[equation]{Definition}
\newtheorem{constr}[equation]{Construction}
\theoremstyle{remark}
\newtheorem{exa}[equation]{Example}
\newtheorem{rem}[equation]{Remark}
\begin{document}
\title{Examples of $2$-unrectifiable normal currents}
\author{Andrea Schioppa}
\address{ETHZ}
\email{andrea.schioppa@math.ethz.ch}
\thanks{The author was supported by the ``ETH Zurich Postdoctoral Fellowship Program and the Marie Curie Actions
  for People COFUND Program''}
\begin{abstract}
We construct new examples of normal (metric) currents using inverse systems of
cube complexes. For any $N\ge 2$ we provide examples of
$N$-dimensional normal currents whose associated vector fields are
simple, and whose supports are purely $2$-unrectifiable and have
Nagata dimension $N$. We show that in $l^\infty$ normal currents can
be realized as limits in the flat distance of currents associated to
cube complexes.
\end{abstract}
\maketitle
\tableofcontents
\section{Introduction}
\label{sec:intro}
\subsection*{Background}
\label{subsec:background}
\def\ccgroup{\mathbb{G}}
Metric currents were introduced by Ambrosio-Kirchheim in \cite{ambrosio-kirch} to
generalize the notion of Federer-Fleming~\cite{federer_fleming_currents} currents to the
metric setting. One motivation to introduce these objects is
formulating / understanding Plateau's problem in metric spaces; there
are also other geometric applications, see for instance
\cite{wenger_filling,wenger_as_rank,wenger_nonpol_nilp}. Lang~\cite{lang_local_currents}
has also formulated a more general
version of the theory in~\cite{ambrosio-kirch} which is  more suitable for some
geometric applications as it does not require the mass of currents to
be finite. In the following we often use the word current to refer to
metric currents, when we refer to Federer-Fleming currents we always
use the term classical current.
\par Unfortunately, as of today there are not many examples of metric
currents. Most of the theory / applications has been developped
looking either at \emph{rectifiable} or \emph{integral currents}, as these objects
admit an alternative (and more concrete) description in the framework
of \emph{rectifiable sets} \cite{ambrosio-rectifiability} in metric spaces. Thus, this work
was in part motivated by the wish to provide new examples of metric currents. 
\par Moreover, it is not even clear what metric currents in Euclidean
spaces are; while the Ambrosio-Kirchheim normal currents coincide in
$\real^n$ with the classical normal currents (of finite mass), for general currents it
is only known that the $1$-dimensional ones are flat
\cite{curr_alb}. While \cite{curr_alb} gives a
geometric description of $1$-currents in metric spaces, when the first
version of this preprint appeared even the
specific question of whether $2$-dimensional normal currents in
$\real^4$ have some special geometric structure was open.
\par In the setting of Carnot groups Williams~\cite{williams-currents} has
obtained a complete classification of normal currents, and also a
partial classification of general metric currents. The study of Carnot
groups actually shows a drawback in the existing definition of metric
currents. In fact, in a non-abelian Carnot group $\ccgroup$ there are
many objects which satisfy  all the axioms in
\cite{ambrosio-kirch,lang_local_currents}
except for the \emph{joint continuity} (Axiom (ii) in Def.~3.1 in
\cite{ambrosio-kirch}; Axiom (2) in Def.~2.1 in \cite{lang_local_currents}). Concretely,  a $k$-dimensional
``current'' $T$ might be obtained using an integral representation
(like for a classical current):
\begin{equation}
  \label{eq:repr}
  T=\vec T\,\mu_{\ccgroup},
\end{equation}
$\vec T$ being a smooth $k$-vector field (in the $k$-th exterior power
of the horizontal distribution), and $\mu_{\ccgroup}$ denoting the Haar measure (here we use Lang's
definition, otherwise just restrict $\mu_{\ccgroup}$ to a set of
finite measure). In general, $T$ is not going to satisfy the joint
continuity axiom: for example, in the first Heisenberg group $X\wedge
Y\,\lebmeas 3.$ would not define a $2$-current in the sense of \cite{ambrosio-kirch,lang_local_currents}.
\par Williams \cite{williams-currents} has also provided examples of
$2$-dimensional normal currents in purely $2$-unrectifiable Carnot
groups. These examples are of the form $(\vec T_1 - \vec
T_2)\,\mu_{\ccgroup}$ where $\vec T_1$ and $\vec T_2$ are constant
simple $2$-vectors chosen to cancel their ``boundaries''. Note again that
the ``currents'' $\vec T_1\,\mu_{\ccgroup}$, $\vec T_2\,\mu_{\ccgroup}$
are not actually currents in the sense of \cite{ambrosio-kirch,lang_local_currents}: this
is unavoidable as \cite{williams-currents} shows that a
$k$-unrectifiable Carnot group cannot admit a $k$-normal current where
the vector field is simple.
\par The arguments in \cite{williams-currents} also show that it is
convenient to work with an integral representation of metric
currents. In \cite{curr_alb} we have showed that any
$k$-current $T$, in the sense of
\cite{ambrosio-kirch,lang_local_currents}, and whose
support is a countable union of doubling metric spaces, admits an
integral representation:
\begin{equation}
  \label{eq:int_rep}
  T=\vec T\,\cmass T.
\end{equation}
where $\cmass T.$ is the mass measure, and $\vec T$ is a $k$-vector
field in the sense of Defn.~\ref{defn:exterior_vecs_forms}. While the
representation~(\ref{eq:int_rep}) is analogous to the classical
setting, a $T$ satisfying~(\ref{eq:int_rep}) is not necessarily going
to be a current in the sense of \cite{ambrosio-kirch,lang_local_currents}, as it might
fail the joint continuity axiom. Having in mind the case of Carnot
groups, in this work we will \emph{drop the joint continuity axiom}
from the definition of metric current, i.e.~we
will define a $k$-current $T$ to be an object admitting a
representation like~(\ref{eq:int_rep}): these objects were called
\emph{precurrents} in \cite{curr_alb} following a terminology
introduced by \cite{williams-currents}. Note that by Theorem \sync
thm:normalprecurr.\ in \cite{curr_alb} the normal currents in the
sense of \cite{ambrosio-kirch,lang_local_currents} and in our extended
sense \emph{coincide}; as in this
work we are essentially concerned with normal currents, dropping the
joint continuity axiom does not cause any inconsistency. However, our
more general notion of current might be of independent interest,
e.g.~in the setting of Carnot groups.
\subsection{Results}
\label{subsec:results}
The Question which motivated this paper is:
\begin{description}
\item[(Q1)] Are there examples of nontrivial normal $N$-dimensional
  currents ($N\ge 2$) whose $N$-vector field is simple and whose
  support is purely $N$-unrectifiable?
\end{description}
In Section~\ref{sec:exas} we
\begin{description}
\item[(Ex)] Exhibit for each $N\ge 2$ an example $N_\infty$ of a
  nontrivial normal $N$-dimensional current whose support $X_\infty$
  is purely $2$-unrectifiable and has Nagata and topological dimension
  $N$. Moreover, the metric measure space $(X_\infty, \cmass
  N_\infty.)$ may be taken to admit a $(1,1)$-Poincar\'e inequality.
\end{description}
These examples are optimal from three perspectives:
\begin{description}
\item[(Prs1)] One cannot have a $1$-unrectifiable support because of
  \cite{paolini_acyclic,paolini_one_normal};
\item[(Prs2)] By a result of Z\"ust \cite{zust_phd} if a space $X$
  supports a nontrivial normal $k$-current it must have Nagata
  dimension at least $k$ (even though the topological dimension might
  be $1$, but not $0$ by \cite{paolini_acyclic,paolini_one_normal});
\item[(Prs3)] The vector field is simple in contrast to the case of
  Carnot groups \cite{williams-currents}.
\end{description}
Essentially, the Nagata dimension \cite{lang_nagata} is a version
of the topological dimension in the Lipschitz category, and thus it is
better suited to handle some questions arising in analysis on metric
spaces, e.g.~questions regarding the extendability of Lipschitz
maps. Note that the topological dimension always bounds from below the
Nagata dimension. By a beautiful Theorem of Buyalo-Lebedeva
\cite{buyalo_nagata} the Nagata dimension and the
topological dimension of a \emph{self-similar} metric space coincide. As an
application one concludes that the Nagata dimension of a Carnot group
coincides with the topological dimension (Carnot groups are
self-similar because they have a familiy of dilations and translations; for another
argument see
\cite{ledonne_rajala_nagata}). 
Note also that by \cite{williams-currents} an
$n$-dimensional non-abelian Carnot group cannot admit a nontrivial
$n$-normal current.
\par The examples $N_\infty$, $X_\infty$ are obtained by relaxing the
requirements of inverse systems $\{(X_i,\mu_i)\}_i$ in \cite{cheeger_inverse_poinc} (see
Defn.~\ref{defn:wad_sys}). In Theorem~\ref{thm:limit_curr_wadsys} we
show how to associate a limit current to the inverse limit $(X_\infty,
\mu_\infty)$ of such a system and also show that $(X_\infty,
\mu_\infty)$ admits a kind of ``calculus'' similar to the one in
$\real^N$ (despite being, in general, purely $2$-unrectifiable).
\par The spaces $\{X_i\}$ are $N$-dimensional cube complexes and
the measures $\{\mu_i\}$ restrict to a constant multiple of Lebesgue
measure on each cell. Moreover, to each $X_i$ one can naturally
associate a normal ``cubical'' current $N_i$, and $N_\infty$ is the
weak limit of the $N_i$. Therefore, a natural question is how general is the idea of
constructing a normal $N$-current as a limit of cubical currents. In
Section~\ref{sec:appx_cube} we show that in $l^\infty$ one can
approximate normal currents by cubical ones in the flat distance (and
hence in the weak topology) while keeping good bounds on the masses.
\par Finally, in analyzing the module of Weaver derivations for these
examples we have found useful some results relating the Nagata dimension
and approximations of $X$ by polyhedra, see Subsection~\ref{subsec:nagata_dim}.
\subsection{Further directions}
\label{subsec:further}
In the first version of this work we asked the following natural question:
\begin{description}
\item[(Q2)] Is it possible to construct in $\real^{n\ge 4}$ a
  $2$-dimensional nontrivial normal current $N$ whose support is
  purely $2$-unrectifiable?
\end{description}
G.~Alberti had told us that such examples do not exist in
$\real^3$. More generally a result of G.~Alberti and A.~Masacesi
shows that in $\real^k$ a \emph{codimension one} normal current can be
represented as an integral of $(k-1)$-rectifiable currents. In a
forthcoming paper, in joint work with U.~Lang, we settle \textbf{(Q2)}
in the negative, the crucial point being that \textbf{(Q2)} is about
a \emph{codimension two} current.
Another natural question is:
\begin{description}
\item[(Q3)] If $(X,\mu)$ is a metric measure space where $X$ has
  Nagata dimension $N$, is the analytic dimension
  (Defn.~\ref{defn:analy_dim}) of $(X,\mu)$ at most $N$?
\end{description}
We provide a positive answer to \textbf{(Q3)},
Theorem~\ref{thm:tower}, under an additional assumption. Note that by
Theorem~\ref{thm:poly_appx} a counterexample to \textbf{(Q3)} would
provide for some $m\ge 1$ a Lipschitz map:
\begin{equation}
  \label{eq:counter1}
  F:X\to\real^{N+m}
\end{equation}
with $dF$ having rank $N+m$ on a set of positive measure, and such
that $F$ can be approximated in the weak* topology (i.e.~pointwise with
uniform bound on the Lipschitz constants) by maps:
\begin{equation}
  \label{eq:counter2}
  \tilde F: X\to\real^{N+m}
\end{equation}
which factor through $N$-dimensional polyhedra. In particular,
$d\tilde F$ would have rank at most $N$ $\mu$-a.e., while being close
to $dF$ in the weak* topology.
\subsection*{Acknowledgements}
\label{subsec:ackwn}
The examples discussed here arose in conversations between B.~Kleiner
and me, where we were looking at the question of producing purely
$2$-unrectifiable higher dimensional versions of the so-called Laakso
spaces. I gratefully thank B.~Kleiner for allowing me to include them
here, and for motivating me to prove Theorem~\ref{thm:cubical_appx}.
\section{Background}
\label{sec:back}
In this Section we recall material on Weaver derivations, Metric
currents, and the Nagata dimension. For Weaver derivations and metric
currents we recall many concepts in a
dry and formulaic style, and refer the interested reader to
\cite{deralb,curr_alb,derivdiff} for more details. For the Nagata
dimension we focus on new results on approximations by polyhedra
(Theorem~\ref{thm:poly_appx}), and finite dimensionality (Theorem~\ref{thm:tower}).
\subsection{Weaver derivations}
\label{subsec:derivations}
\par For more information we refer the reader to \cite{weaver00,deralb}.
An \textbf{$L^\infty(\mu)$-module} $M$ is a Banach space $M$ which
is also an $L^\infty(\mu)$-module and such that for all
$(m,\lambda)\in M\times L^\infty(\mu)$ one has:
\begin{equation}
  \label{eq:boundedaction}
  \|\lambda m\|_M\le\|\lambda\|_{L^\infty(\mu)}\,\|m\|_M.
\end{equation}
Among $L^\infty(\mu)$-modules a special r\^ole is played by
\textbf{$L^\infty(\mu)$-normed modules}:
\begin{defn}[Normed modules]
  \label{defn:local_norm}
  An $L^\infty(\mu)$-module $M$ is said to be an
  \textbf{$L^\infty(\mu)$-normed module} if there is a map
  \begin{equation}
    |\cdot|_{M,{\text{loc}}}:M\to L^\infty(\mu) 
  \end{equation}
such that:
\begin{enumerate}
\item For each $m\in M$ one has $|m|_{M,{\text{loc}}}\ge0$;
\item For all
  $c_1,c_2\in\real$ and $m_1,m_2\in M$ one has:
  \begin{equation}
        |c_1m_1+c_2m_2|_{M,{\text{loc}}}\le|c_1||m_1|_{M,{\text{loc}}}+|c_2||m_2|_{M,{\text{loc}}};
  \end{equation}
\item For each $\lambda\in L^\infty(\mu)$ and each $m\in M$, one has:
  \begin{equation}
    |\lambda m|_{M,{\text{loc}}}=|\lambda|\,|m|_{M,{\text{loc}}};
  \end{equation}
\item The local seminorm $|\cdot|_{M,{\text{loc}}}$ can be used to
  reconstruct the norm of any $m\in M$:
  \begin{equation}
    \|m\|_M=\|\,|m|_{M,{\text{loc}}}\,\|_{L^\infty(\mu)}.
  \end{equation}
\end{enumerate}
\end{defn}
Let $\lipalg X.$ denote the algebra of bounded real-valued Lipschitz
functions defined on $X$. This is a Banach algebra with the norm the
max of the sup norm $\|f\|_\infty$ and the Lipschitz constant $\glip f.$ (see \cite[Sec.~1.6]{weaver_book99}).
\begin{defn}[Weaver derivation]
  \label{defn:derivations}
    A \textbf{derivation $D:\lipalg X.\to L^\infty(\mu)$} is a weak*
    continuous, bounded linear map satisfying the product rule:
    \begin{equation}
      D(fg)=fDg+gDf.
    \end{equation}
A sequence $f_n\to f$ in weak* topology on $\lipalg X.$ if there is a
uniform bound on the global Lipschitz constants $\glip f_n.$ of the
$\{f_n\}$ and if $f_n\to f$ pointwise.
\par The collection of all derivations $\wder\mu.$ is an
  $L^\infty(\mu)$-normed module
  \cite{weaver00}
  and the corresponding  local norm will be denoted by
$\locnorm\,\cdot\,,{\wder\mu.}.$. Note also that $\wder\mu.$ depends
only on the measure class of $\mu$.
\end{defn}
\begin{exa}
  \label{exa:weaver_der}
  Consider $(\real^n,\lebmeas n.)$ and let $\partial_\alpha$ be the
  partial derivative in the $\alpha$-direction. By Rademacher's Theorem
  $\partial_\alpha$ defines a bounded linear map
  $\partial_\alpha:\lipalg \real^n.\to L^\infty(\lebmeas n.)$. The weak*
  continuity can be reduced to the $1$-dimensional case (and hence to
  integration by parts) using Fubini's Theorem.
\end{exa}
\begin{defn}[Submodules and locality]
  \label{defn:submodules}
  Consider a Borel set $U\subset X$ and a derivation
  $D\in\wder\mu\on U.$. The derivation $D$ can be also regarded as
  an element of $\wder\mu.$ by extending $Df$ to be $0$ on $X\setminus
  U$. In particular,
  the module $\wder\mu\on U.$ can be naturally identified with the
  submodule $\chi_U\wder\mu.$ of $\wder\mu.$.
  \par By Lemma \sync lem:locality.\ in \cite{weaver00} derivations are \textbf{local} in the following sense:
  if $U$ is $\mu$-measurable and if $f,g\in\lipalg X.$ agree on $U$,
  then for each $D\in\wder\mu.$, $\chi_UDf=\chi_UDg$. Note that locality allows to extend the action of derivations on
  Lipschitz functions $f$, so that $Df$ is well-defined.
\end{defn}
\begin{defn}[Analytic dimension]
  \label{defn:analy_dim}
  We define the \textbf{analytic dimension} of the metric measure
  space $(X,\mu)$ to be the \textbf{index} of $\wder\mu.$:
  \begin{multline}
    \label{eq:analy_dim_1}    
    \text{index of $\wder\mu.$} = \sup\{
    n\in\natural: \text{$\exists U$ Borel: $\wder\mu\on U.$ contains
      $n$-independent}\\ \text{elements (over $L^\infty(\mu\on U)$)
    }
    \}.
  \end{multline}
  If either $\spt\mu$ or $X$ are doubling (see~\cite{deralb}) then $\wder\mu.$ has finite index;
  moreover, if $\wder\mu.$ has finite index, it can be decomposed into a
  direct sum of free submodules (over smaller rings), see \cite{derivdiff}.
\end{defn}
\begin{exa}[Index in $\real^n$]
  \label{exa:poly_finite}
  Let $\mu$ be a Radon measure on $\real^n$; by the
  Stone-Weierstrass Theorem for Lipschitz algebras~\cite[Thm.~4.1.8]{weaver_book99}
  the polynomials in $\{x_1,\cdots,x_n\}$ (after truncating the
  polynomials to have bounded absolute value, e.g.~postcomposing with
  $\max(\cdot, K)$,
  $\min(\cdot, K)$) are weak*-dense in $\lipalg \real^n.$ and so the index of
  $\wder\mu.$ is at most $n$.
\end{exa}
A standard way to produce derivations is to use Alberti
representations. We deal here with a more restrictive situation, for
the general case see  \cite{deralb}.
\begin{defn}[Alberti representations]
  \label{defn:alb_rep}
  Let $\curves(X)$ denote the space of Lipschitz curves in $X$
  topologized with the Fell topology \cite[(12.7)]{kechris_desc} on their graphs.
  Let $\mu$ be a Radon measure on $X$. An \textbf{Alberti representation} of
  $\mu$ is a pair $\albrep.=[Q,w]$ where $Q$ is a Radon measure on
  $\curves(X)$ and $w$ a Borel function $w:X\to[0,\infty)$ such that:
  \begin{equation}
    \label{eq:alb_rep_1}
    \mu=\int_{\curves(X)}w\cdot\hmeas._\gamma\,dQ(\gamma),
  \end{equation}
  where $\hmeas._\gamma$ is the length measure on
  $\gamma$, and the integral is interpreted in the weak* sense. We say that
  $\albrep.$ is \textbf{$C$-Lipschitz} (resp.~\textbf{$[C,D]$-biLipschitz}) if $Q$ is concentrated on the
  set of $C$-Lipschitz (resp.~$[C,D]$-biLipschitz) curves. 
  \par  Let $\albrep.=[Q,w]$ be a $C$-Lipschitz Alberti representation of a measure
  $\nu\ll\mu$; then the formula:
  \begin{multline}
    \label{eq:alberti_to_derivation_1}
    \int_X gD_{\albrep.}f\,d\nu =
    \int_{\curves(X)}dQ(\gamma)\int wg\cdot
    \partial_\gamma f\,d\hmeas._\gamma\\
    (\forall (g,f)\in C_c(X)\times\lipalg X.),
  \end{multline}
  where $\partial_\gamma f$ denotes the derivative of $f$ along $\gamma$,
  defines a derivation $D_{\albrep.}\in\wder\nu.\subset\wder\mu.$ with
  $\|D_{\albrep.}\|_{\wder\mu.}\le C$.
\end{defn}
\begin{defn}[Weaver differentials]
  \label{defn:module_forms}
  The \textbf{module of $1$-forms} $\wform\mu.$ is the dual module of
  $\wder\mu.$, i.e.~it consists of the bounded module homomorphisms
  $\wder\mu.\to L^\infty(\mu)$. The module $\wform\mu.$ is an
  $L^\infty(\mu)$-normed module and the local norm will be denoted by
  $\locnorm\,\cdot\,,{\wform\mu.}.$.
  \par To each $f\in\lipalg X.$ one can associate the $1$-form, its
  \textbf{differential}
  $df\in\wform\mu.$, by letting:
  \begin{equation}
    \label{eq:module_forms1}
    \langle df,D\rangle=Df\quad(\forall D\in\wder\mu.);
  \end{equation}
  the map $d:\lipalg X.\to\wform\mu.$ is a weak* continuous $1$-Lipschitz
  linear map satisfying the product rule $d(fg)=gdf+fdg$. Note that by
  locality (Defn.~\ref{defn:submodules}) one can
  extend the domain of $d$ to the set of Lipschitz functions so that if $f$ is Lipschitz,
  $df$ is a well-defined element of $\wform\mu.$ and
  $\|df\|_{\wform\mu.}\le\glip f.$.
\end{defn}
\begin{defn}[Push-forward / Pull-back]
  \label{defn:push_pull}
  Let $F:X\to Y$ be Lipschitz and $\mu$ a Radon measure on $X$ such
  that $F_{\#}\mu$ is a Radon measure on $Y$. The
  \textbf{push-forward} map
  \begin{equation}
    \label{eq:push_pull_1}
    F_{\#}:\wder\mu.\to\wder F_{\#}\mu.
  \end{equation}
  associates to $D\in\wder\mu.$ the unique $F_{\#}D\in\wder F_{\#}\mu.$
  such that:
  \begin{multline}
    \label{eq:push_pull_2}
    \int_X g\circ F\, D(f\circ F)\,d\mu = \int_Y g \, (F_{\#}
    D)f\,dF_{\#}\mu \\(\forall (g,f)\in C_c(Y)\times\lipalg Y.).
  \end{multline}
  The dual map of $F_{\#}$ is the \textbf{pull-back}:
  \begin{equation}
    \label{eq:push_pull_3}
    \begin{aligned}
      F^{\#}:\wform F_{\#}\mu.&\to\wform\mu.\\
      df&\mapsto d(f\circ F).
    \end{aligned}
  \end{equation}
\end{defn}
\begin{defn}[Exterior powers]
  \label{defn:exterior_vecs_forms}
  Similarly as for vector fields and differential forms, it is
  possible to define the exterior powers $\hwder\mu,.$ and
  $\hwform\mu,.$ (see Definition \sync def:ext.\ and Remark \sync
  rem:ext. \ in \cite{curr_alb}). We also let
  $\hwder \mu,0.=\hwform \mu,0.=L^\infty(\mu)$. Properties of
  $\hwder\mu,.$ and $\hwform\mu,.$ that we are going to use are:
  \begin{description}
  \item[(Ex1)] $\hwder\mu,.$ and $\hwform\mu,.$ are
    $L^\infty(\mu)$-normed modules and finite linear combinations of
    $k$-fold exterior powers (the simple vectors in the classical
    algebraic sense) such as
    \begin{equation}
      \label{eq:exterior_vecs_forms_1}
      D_1\wedge\cdots\wedge D_k,\quad df_1\wedge\cdots \wedge df_k,
    \end{equation}
    are dense;
  \item[(Ex2)] If $\wder\mu.$ or $\wform\mu.$ are finitely generated,
    the $k$-fold exterior products of the generators provide a
    generating set;
  \item[(Ex3)] There are exterior products
    \begin{equation}
      \label{eq:exterior_vecs_forms_2}
      \begin{aligned}
        \wedge:\hwder\mu,.\times\hwder\mu,l.&\to\hwder\mu,k+l.\\
        \wedge:\hwform\mu,.\times\hwform\mu,l.&\to\hwform\mu,k+l.
      \end{aligned}
    \end{equation}
    which are bilinear and have norm at most $1$;
  \item[(Ex4)] There is a natural bilinear pairing:
    \begin{equation}
      \label{eq:exterior_vecs_forms_3}
      \langle\cdot,\cdot\rangle:\hwder\mu,.\times\hwform\mu,k.\to L^\infty(\mu)
    \end{equation}
    which satisfies:
    \begin{equation}
      \label{eq:exterior_vecs_forms_4}
      \begin{aligned}
        \langle D_1\wedge\cdots \wedge D_k, df_1\wedge\cdots\wedge
        df_k\rangle &= \det(D_if_j)_{i,j}\\
        \left|
          \langle\vec\xi,\omega\rangle
        \right|&\le k!
        |\vec\xi|_{\hwder\mu,.}\,|\omega|_{\hwform\mu,.}\quad(\forall
        (\vec\xi,\omega)\in
        \hwder\mu,.\times\hwform\mu,k.);
      \end{aligned}
    \end{equation}
  \item[(Ex5)] Given $\vec\xi\in\hwder\mu,.$, $\omega\in\hwform
    \mu,m.$ for $m\le k$ the \textbf{interior product}
    $\vec\xi\on\omega\in\hwder\mu,k-m.$ is defined so that for each
    $\tilde\omega\in\hwform\mu,k-m.$ one has:
    \begin{equation}
      \label{eq:exterior_vecs_forms_5}
      \langle\vec\xi\on\omega, \tilde\omega\rangle=\langle\vec\xi,\omega\wedge\tilde\omega\rangle.
    \end{equation}
  \end{description}
  \end{defn}
\subsection{Metric currents}
\label{subsec:metric_currs}
\begin{defn}[Metric currents]
  \label{defn:metric_currents}
  A \textbf{$k$-dimensional metric current} $T$ in $X$ is a pair $(\mu,\vec T)$
  where $\mu$ is a Radon measure on $X$ and $\vec
  T\in\hwder\mu,.$. Given $\omega\in\hwform\mu,.$ with integrable
  local norm, i.e.~$|\omega|_{\hwform\mu,.}\in L^1(\mu)$, we let:
  \begin{equation}
    \label{eq:metric_currents_1}
    T(\omega)=\int_X\langle\vec T,\omega\rangle\,d\mu.
  \end{equation}
  If $f_0,\cdots,f_k$ are Lipschitz functions such that
  \begin{equation}
    \label{eq:metric_currents_2}
    \omega=f_0df_1\wedge\cdots\wedge df_k
  \end{equation}
  has $\mu$-integrable local norm, we just let:
  \begin{equation}
    \label{eq:metric_currents_3}
    T(f_0,f_1,\cdots,f_k)=T(f_0df_1\wedge\cdots\wedge df_k).
  \end{equation}
  The \textbf{mass measure} of $T$ is $\cmass T.=|\vec T|_{\hwder\mu,.}\mu$, and
  if this measure is finite, then $T$ \textbf{has finite mass}; in this case
  the \textbf{mass-norm} of $T$ is:
  \begin{equation}
    \label{eq:metric_currents_4}
    \mcnrm T.=\int_X |\vec T|_{\hwder\mu,.}\,d\mu.
  \end{equation}
  The \textbf{support} $\spt T$ of $T$ is the support of $\cmass T.$.
\end{defn}
\begin{defn}[Boundary and normality]
  \label{defn:boundary}
  Let $T$ be a metric current and $\{f_i\}_{i=0}^k\subset\lipalg X.$; we define:
  \begin{equation}
    \label{eq:boundary_1}
    \partial T(f_0,f_1,\cdots, f_{k-1}) = T(df_0\wedge df_1\wedge\cdots\wedge df_{k-1});
  \end{equation}
  if there is a Radon measure $\nu$ such that whenever
  \begin{equation}
    \label{eq:boundary_2}
    |df_0\wedge df_1\wedge\cdots\wedge df_{k-1}|_{\hwform \mu,k.}\in L^1(\mu)
  \end{equation}
  one has:
  \begin{equation}
    \label{eq:boundary_3}
    \left|\partial T(f_0, f_1, \cdots, f_{k-1})\right|
    \le\prod_{i=1}^{k-1}\glip f_i.\,\int_X|f_0|\,d\nu,
  \end{equation}
  then \textbf{the boundary of} $\partial T$ of $T$ is still a metric current, i.e.~one can find
  $\vec{\partial T}\in\hwder\nu,k-1.$ such that:
  \begin{equation}
    \label{eq:boundary_4}
    \partial T(f_0, f_1, \cdots, f_{k-1}) = \int_X
    f_0\langle\vec{\partial T},df_1\wedge\cdots\wedge df_{k-1}\rangle\,d\nu.
  \end{equation}
  \par A metric current $T$ whose boundary is still a current is called
  \textbf{normal}; if both $\cmass T.$ and $\cmass\partial T.$ are
  finite measures we let the \textbf{normal mass-norm} be:
  \begin{equation}
    \label{eq:boundary_5}
    \ncnrm T.=\mcnrm T.+\mcnrm\partial T..
  \end{equation}
  \par A normal current $T$ such that $\cmass\partial T.$ is locally
  finite satisfies the
  following \textbf{joint continuity axiom} (Theorem \sync thm:normalprecurr.\ in \cite{curr_alb}): if
  $\{f_{i,n}\}_{i=0,\cdots,k; n\in\natural\cup\{\infty\}}$ satisfy:
  \begin{equation}
    \label{eq:boundary_6}
    f_{i,n}\xrightarrow{\text{w*}} f_{i,\infty}\quad(\text{$\forall
      i,$ as $
      n\to\infty)$}
  \end{equation}
  and if the measures
  \begin{equation}
    \label{eq:boundary_6bis}    
      |f_{0,n}df_{1,n}\wedge\cdots\wedge df_{k,n}|_{\hwder{\cmass
          T.},k.}\,\cmass T.\quad(n\in\natural\cup\{\infty\})      
    \end{equation}
    are tight,
    then:
    \begin{equation}
      \label{eq:boundary_7}
    \lim_{n\to\infty}T(f_{0,n},f_{1,n},\cdots,f_{k,n})=T(f_{0,\infty},f_{1,\infty},\cdots,f_{k,\infty}).
  \end{equation}
\end{defn}
\begin{defn}[Push-forward, interior product]
  \label{defn:push_product}
  Let $F:X\to Y$ be Lipschitz and $\mu$ a Radon measure on $X$ such
  that $F_{\#}\mu$ is a Radon measure on $Y$. The
  map $  F_{\#}:\wder\mu.\to\wder F_{\#}\mu.$ induces maps
  \begin{equation}
    \label{eq:push_product_1}
    \begin{aligned}
      F_{\#}:\hwder\mu,.&\to\hwder F_{\#}\mu,.\\
      F^{\#}:\hwform F_{\#}\mu,.&\to\hwform \mu,..
    \end{aligned}
  \end{equation}
  If $T=(\mu,\vec T)$ we let $F_{\#} T$ denote the
  \textbf{push-forward}:
  \begin{equation}
    \label{eq:push_product_2}
    F_{\#} T = (F_{\#}\mu, F_{\#}\vec T).
  \end{equation}
  If $\omega\in\hwform\mu,m.$ we let:
  \begin{equation}
    \label{eq:push_product_3}
    T\on\omega = (\mu, \vec T\on\omega).
  \end{equation}
\end{defn}
\begin{defn}[Weak topology]
  \label{defn:weak_topo}
  We say that a sequence of $k$-dimensional currents $\{T_n\}$
  converges to a $k$-dimensional current
  $T$ in the \textbf{weak topology} if whenever $\{f_i\}_{i=0}^k$ are
  Lipschitz functions such that for each $i,n$ $\spt f_i\cap \spt T_n$
  is compact, one has:
  \begin{equation}
    \label{eq:weak_topo_1}
    \lim_{n\to\infty}T_n(f_0,f_1,\cdots,f_k)=T(f_0,f_1,\cdots,f_k).
  \end{equation}
  \par The \textbf{flat norm} of a current $T$ is:
  \begin{equation}
    \label{eq:weak_topo_2}
    \flatnrm T. = \inf\{\mcnrm S_1. + \mcnrm S_2.:\text{$T= S_1
      + \partial S_2$ for $S_1$, $S_2$ currents}\}.
  \end{equation}
  In particular, if $\flatnrm T_n.\to0$, then $T_n\to 0$ in the weak topology.
\end{defn}
\subsection{Nagata dimension}
\label{subsec:nagata_dim}
\def\faset#1.{\setbox1=\hbox{$#1$\unskip}\mathcal{C}_{\ifdim\wd1>0pt #1\else
    i\fi}}
\begin{defn}[Nagata cover]
  \label{defn:naga_cov}
  Let $(C,s,N)\in(0,\infty)^2\times(\natural\cup\{0\})$; a
  \textbf{$(C,s,N)$-Nagata cover} of $X$ is a collection of $(N+1)$-families
  of sets $\{\faset.\}_{i=0,\cdots,N}$ such that:
  \begin{description}
  \item[(NSep)] If $A,B\in\faset.$ are distinct then:
    \begin{equation}
      \label{eq:naga_cov_1}
      \dist(A,B)\ge s;
    \end{equation}    
  \item[(NBd)] For each $A\in\faset.$:
    \begin{equation}
      \label{eq:naga_cov_2}
      \diam A\le Cs.
    \end{equation}
  \end{description}
A Nagata cover is \textbf{sorted} if $i\ge1$ and $B\in\faset.$ imply that for
some $A\in\faset i-1.$ one has
\begin{equation}
  \label{eq:naga_cov_3}
  \dist(A,B)<s.
\end{equation}
Note that from any Nagata cover one can produce a sorted one by moving
sets across the families $\{\faset.\}_{i=0,\cdots,N}$.
\end{defn}
\begin{defn}[Nagata dimension]
  \label{defn:nag_dim}
  A metric space $X$ has \textbf{Nagata dimension} at most $N$ if
  there is a $C>0$ (cover-separation parameter) such that for each $s>0$ $X$ admits a
  $(C,s,N)$-Nagata cover. The Nagata dimension of $X$ is the smallest
  $N$ so that $X$ has Nagata dimension at most $N$.
  \par A metric space $X$ has \textbf{small Nagata dimension} at most
  $N$ if for each $x\in X$ there is an $r>0$ (scale parameter) such that $B(x,r)$ has
  Nagata dimension at most $N$. Note that it might happen that the scale
  $r$ might be chosen uniformly, i.e.~independently of $x$. The small 
  Nagata dimension of $X$ is the smallest
  $N$ so that $X$ has small Nagata dimension at most $N$.
\end{defn}
\def\pdec{_{\text{\normalfont p}}}
\begin{thm}[Polyhedral approximation]
  \label{thm:poly_appx}
  $X$ has Nagata dimension at most $N$ if and only if there is a
  constant $C\pdec$ depending only on the parameters $(C,N)$ in the
  definition of Nagata dimension, such that for each $s>0$ there is an
  $N$-dimensional simplicial complex $P$, equipped with a metric $d_P$
  which restricts on each simplex to a metric induced by a Eucidean
  norm, and a $C\pdec$-Lipschitz map
  \begin{equation}
    \label{eq:poly_appx_s1}
    F:X\to P
  \end{equation}
  such that
  \begin{equation}
    \label{eq:poly_appx_s2}
    \|F^*d_P-d_X\|_\infty\le C\pdec s.
  \end{equation}
\end{thm}
\begin{proof}
  Sufficiency follows from \cite[Prop.~2.5]{lang_nagata}; we focus on
  necessity.
  \par\noindent\texttt{Step 1: Construction of $P$ and $F$.}
  \par Take a sorted $(C,s,N)$-Nagata cover
  $\{\faset.\}_{i=0,\cdots,N}$; to each $S_i\in\faset.$ associate a
  ($3s^{-1}$)-Lipschitz function:
  \begin{equation}
    \label{eq:poly_appx_p1}
    \Phi_{S_i}:X\to[0,1]
  \end{equation}
  such that:
  \begin{equation}
    \label{eq:poly_appx_p2}
    \Phi_{S_i} =
    \begin{cases}
      0 &\text{on $X\setminus B(S_i,s/3)$}\\
      1 &\text{on $S_i$.}
    \end{cases}
  \end{equation}
  By \textbf{(NSep)} for each $x\in X$ ane each $i$ there is at most one
  $S_i\in\faset.$ such that $\Phi_{S_i}(x)\ne0$ and so
  \begin{equation}
    \label{eq:poly_appx_p3}
    1\le\sum_{i=0}^N\sum_{S_i\in\faset.}\Phi_{S_i}\le N+1,
  \end{equation}
  and we can rescale the $\Phi_{S_i}$ by their sum to get a
  ($C(N)s^{-1}$)-Lipschitz partition of unity (still denoted by
  $\{\Phi_{S_i}\}$):
  \begin{equation}
    \sum_{i=0}^N\sum_{S_i\in\faset.}\Phi_{S_i}(x)=1.
  \end{equation}
  \par Let $\mathcal{S}$ denote the metric space whose points are the
  $\{S_i\in\faset.\}_{i=0,\cdots,N}$ and whose distance is the Hausdorff
  distance. We embedd $\mathcal{S}$ in $l^\infty$ and let $[S_i]$ denote
  the image of $S_i$. We define
  \begin{equation}
    \label{eq:poly_appx_p4}
    \begin{aligned}
      F:X&\to l^\infty\\
      x&\mapsto\sum_{i=0}^N\sum_{S_i\in\faset.}\Phi_{S_i}(x)[S_i],
    \end{aligned}
  \end{equation}
  and note that $F(X)\subset P$, where $P$ is the $N$-dimensional
  simplicial complex obtained by taking the convex hull of all finite
  tuples $([S_{i_0}],\cdots,[S_{i_k}])$ whenever
  \begin{equation}
    \label{eq:poly_appx_p5}
    \Phi_{S_{i_0}}\cdot\Phi_{S_{i_1}}\cdots\Phi_{S_{i_k}}\ne0.
  \end{equation}
  Note that the metric $d_P$ is the restriction of the ambient metric
  of $l^\infty$.
  \par\noindent\texttt{Step 2: Proof of~{\normalfont (\ref{eq:poly_appx_s2})}.}
  \par For $x\in X$ we let $L(x)=\{S_i: \Phi_{S_i}(x)\ne0\}$. If
  $S_x,S\in L(x)$ then:
  \begin{equation}
    \label{eq:poly_appx_p6}
    \dist(S, S_x)\le\frac{2}{3}s,
  \end{equation}
  and so we have the bound on the Hausdorff-distance:
  \begin{equation}
    \label{eq:poly_appx_p7}
    d_H(S,S_x)\le(\frac{2}{3}+C)s.
  \end{equation}
  If $S_x\in L(x)$, $S_y\in L(y)$ we then have:
  \begin{align}
    \label{eq:poly_appx_p8}
    \|F(x)-[S_x]\|_{l^\infty}&\le(\frac{2}{3}+C)s\\
    \label{eq:poly_appx_p9}
    \|F(y)-[S_y]\|_{l^\infty}&\le(\frac{2}{3}+C)s.
  \end{align}
  As $d(x,S_x)\le s/3$, $d(y, S_y)\le s/3$ we conclude that
  \begin{equation}
    \label{eq:poly_appx_p10}
    \left| d(x,y) - d_H(S_x,S_y)\right|\le 2(\frac{1}{3}+C)s,
  \end{equation}
  from which we get:
  \begin{equation}
    \label{eq:poly_appx_p11}
    \left|
      \|F(x)-F(y)\|_{l^\infty}-d(x,y)
    \right|\le (2+4C)s.
  \end{equation}
 \par\noindent\texttt{Step 3: Uniform Lipschitz bound on $F$.}
 \par If $L(x)\cap L(y)=\emptyset$ then $d(x,y)\gtrsim s$ and so a
 uniform bound on:
 \begin{equation}
   \label{eq:poly_appx_p12}
   \frac{\|F(x)-F(y)\|_{l^\infty}}{d(x,y)}
 \end{equation}
 follows from~(\ref{eq:poly_appx_s2}).
 \par Choose $W_{x,y}\in L(x)\cap L(y)$; then:
 \begin{equation}
   \label{eq:poly_appx_p13}
   \begin{split}
     F(x) - F(y) &= \sum_{S\in L(x)}\Phi_S(x)[S] - \sum_{T\in
       L(y)}\Phi_T(y)[T]\\
     &= \sum_{S\in L(x)}\Phi_S(x)[S]-\sum_{S\in L(x)}\Phi_S(x)[W_{x,y}]\\
     &\mskip\munsplit + \sum_{T\in L(y)}\Phi_T(y)[W_{x,y}] - \sum_{T\in
       L(y)}\Phi_T(y)[T]\\
     &=\sum_{S\in L(x)\setminus\{W_{x,y}\}}\Phi_S(x)([S]-[W_{x,y}]) -
     \sum_{T\in L(y)\setminus\{W_{x,y}\}}\Phi_T(y)([T]-[W_{x,y}])\\
     &=\sum_{S\in L(x)\setminus L(y)}(\Phi_S(x)-\Phi_S(y))([S]-[W_{x,y}])
     -\sum_{T\in L(y)\setminus L(x)}(\Phi_T(y)-\Phi_T(x))([T]-[W_{x,y}])\\
     &\mskip\munsplit +\sum_{W\in L(x)\cap L(y)}(\Phi_W(x)-\Phi_W(y))([W]-[W_{x,y}]).
   \end{split}
 \end{equation}
 Now if $Z\in L(x)\cup L(y)$ and $W\in L(x)\cap L(y)$ we have:
 \begin{equation}
   \label{eq:poly_appx_p14}
   \begin{aligned}
     \|[Z] - [W_{x,y}]\|_{l^\infty}&\le(\frac{2}{3}+C)s\\
     \|[W] - [W_{x,y}]\|_{l^\infty}&\le(\frac{2}{3}+C)s.
   \end{aligned}
 \end{equation}
 Then by~(\ref{eq:poly_appx_p13}):
 \begin{equation}
   \label{eq:poly_appx_p15}
   \begin{split}
     \|F(x)-F(y)\|_{l^\infty}&\le\frac{3}{s} d(x,y)\left(
       \#(L(x)\setminus L(y)) +  \#(L(y)\setminus L(x))
       +\#(L(x)\cap L(y))
     \right) (\frac{2}{3}+C)s\\
     &\le 6N(\frac{2}{3}+C)d(x,y).
   \end{split}
 \end{equation}
\end{proof}
\begin{defn}[\textbf{(TAP)}($N$)]
  \label{defn:tower}
  A metric space has the property \textbf{tower of approximations by
    $N$-dimensional polyhedra} (abbr.~\textbf{(TAP)}($N$)) if there
  are constants $C$, $\{C_n\}$, $s_n\searrow0$, and $N$-dimensional
  polyhedral complexes $P_n$ (where the metric restricts on each
  simplex to a metric induced by a norm), and $C$-Lipschitz maps:
  \begin{equation}
    \label{eq:tower_1}
    F_n:X\to P_n
  \end{equation} such that:
  \begin{align}
    \label{eq:tower_2}
    \|F_n^*d_{P_n}-d_X\|_\infty&\le Cs_n\\
    \label{eq:tower_3}
    F_n^*d_{P_n}&\le C_nF_{n+1}^*d_{P_{n+1}}.
  \end{align}
Equivalently, (\ref{eq:tower_3}) can be reformulated by asking for
$C_n$-Lipschitz maps $\pi_n:P_{n+1}\to P_n$ which make the following
diagram commute:
\begin{equation}
  \label{eq:tower_4}
  \xy
(0,0)*+{X}="x";
(20,10)*+{P_n}="pn";
(20,-10)*+{P_{n+1}}="pnn";
{\ar "x"; "pn"}?*!/_4mm/{F_n};
{\ar "x"; "pnn"}?*!/^4mm/{F_{n+1}};
{\ar "pnn"; "pn"}?*!/^4mm/{\pi_n};
\endxy
\end{equation}
\end{defn}
\begin{thm}[finite-dimensionality from \textbf{(TAP)}($N$)]
  \label{thm:tower}
  Let $X$ have property \textbf{(TAP)}($N$) and $\mu$ be a Radon
  measure on $X$. Then the analytic dimension of $(X,\mu)$ is at most $N$.
\end{thm}
\begin{proof}
  \par\noindent\texttt{Step 1: Weak* approximation by Lipschitz
    maps.}
  \par Let $f:X\to\real$ be $1$-Lipschitz. Then there is a function
  $\tilde f: P_n\to\real$ with Lipschitz constant $\glip\tilde f.\le
  L(C)$ such that:
  \begin{equation}
    \label{eq:tower_p1}
    \|f-\tilde f\circ F_n\|_\infty\le L(C)s_n.
  \end{equation}
  In fact, it suffices to select a maximal $4Cs_n$-separated net
  $\mathcal{S}$ in $X$, let $\tilde f(F_n(x))= f(x)$ for
  $x\in\mathcal{S}$ and extend $\tilde f$ by MacShane's Lemma.
  \par\noindent\texttt{Step 2: Mazur's Lemma}
\par Assume that $\mu$ is a finite Borel measure on $X$ and 
\begin{equation}
  \label{eq:tower_p2}
  \{D_1,\cdots, D_k\}\subset\wder\mu.
\end{equation}
are independent so that there are $1$-Lipschitz functions
\begin{equation}
  \label{eq:tower_p3}
  \{g_1,\cdots, g_k\}
\end{equation}
such that the matrix $(D_ig_j)$ has $\mu$-a.e.~rank $k$. By
\texttt{Step 1} each $g_j$ can be approximated in the weak* topology
by a sequence $g_j^{(n)}\circ F_n$ where $g_j^{(n)}: P_n\to\real$.
\par As $D_i(g_j^{(n)}\circ F_n)\xrightarrow{\text{w*}} D_ig_j$, using
Mazur's Lemma, we can find finite convex linear combinations
$\sum_{n}t^{(m)}_ng_j^{(n)}$ such that:
\begin{equation}
  \label{eq:tower_p4}
  D_i\left(
    \sum_nt^{(m)}_ng_j^{(n)}\circ F_n
  \right)\xrightarrow{L^2(\mu)}D_ig_j,
\end{equation}
where note that $t^{(m)}_n$ does not depend on $i$.
\par In particular, for $m$ sufficiently large we can find a set
$K\subset X$ of positive measure on which:
\begin{equation}
  \label{eq:tower_p5}
  \left|
    \det \left(
      D_i\left(
        \sum_nt^{(m)}_ng_j^{(n)}\circ F_n
      \right)
    \right)_{i,j}
  \right|>0;
\end{equation}
as the sum $\sum_nt^{(m)}_ng_j^{(n)}$ is finite, and as by
\textbf{(TAP)}($N$) any $\psi:P_n\to\real$ can be written as
$\psi=\tilde\psi\circ\pi_n$ where $\pi_n:P_{n+1}\to P_n$ is
$C_n$-Lipschitz, we can find $n$ and Lipschitz functions $\tilde
g_j:P_n\to\real$ such that:
\begin{equation}
  \label{eq:tower_p6}
  \left|
    \det (D_i(\tilde g_j\circ F_n))_{i,j}
  \right|>0,
\end{equation}
on a subset $\tilde K$ of positive measure.
\par If $\tilde\mu$ denotes the disintegration of $\mu\on \tilde K$
with respect to $F_{n\#}\mu\on\tilde K$, then for $F_{n\#}\mu\on\tilde
K$-a.e.~$p$ $D_i(g_j\circ F_n)$ is $\tilde\mu(p)$-a.e.~constant on
$F_n^{-1}(p)$. Thus the $d\tilde g_j$ are independent in $\wform
F_{n\#}\mu\on\tilde K.$; as $F_{n\#}\mu\on\tilde K$ is a Radon measure
supported in an $N$-dimensional simplicial complex, we must have $k\le
N$ by Example~\ref{exa:poly_finite}.
\end{proof}
\section{Inverse Systems}
\label{sec:inv_sys}
In this Section we first discuss inverse systems of cube complexes
which yield Poincar\'e inequalities. Definition~\ref{defn:ads_sys} and
Theorem~\ref{thm:adsys_pi} are contained in
\cite[Sec.~11]{cheeger_inverse_poinc}; note that in general the
constant in the Poincar\'e inequality might depend both on the location of
the ball $B$ in the space and its radius $\brad(B)$, because the geometry at $\infty$ of such cube
complexes might be complicated. We then consider inverse systems
which satisfy a relaxed set of axioms, Definition \ref{defn:wad_sys},
which is suitable for constructing normal currents
(Defn.~\ref{defn:normal_sys}). The description of how to produce such
currents and the kind of ``calculus'' supported by such spaces is
discussed in Theorem~\ref{thm:limit_curr_wadsys}.
\par For a cube complex $X$ we will let $\ccell X,k.$ denote the set
of its $k$-dimensional cells.
\def\gmdec{_\text{\normalfont geo}}
\def\gldec{_\text{\normalfont gall}}
\def\lkstar{\text{\normalfont St}}
\begin{defn}[Admissible inverse systems / AIS]
  \label{defn:ads_sys}
  Let $(N,m)\in \natural\times(\natural\cap[2,\infty))$ and consider a
  collection of metric measure spaces $\{(X_i,\mu_i)\}_{i\in
    I}$ and maps $\{\pi_i\}_{i\in I}$ where the index set $I$ is
  of the form $\{k\in\zahlen: k\ge k_0\}$. This
  collection $\adsys$ is an \textbf{($N$-dimensional) admissible
    inverse system} if the following axioms hold.
  \begin{description}
  \item[(IBGeom)] Each $X_i$ is a nonempty connected cube-complex (with
    the length metric) which is a union of its $N$-dimensional cells
    which are isometric to the Euclidean cube $[0,m^{-i}]^N$;
    moreover, there is a uniform bound $C\gmdec$ on the cardinality of
    each link.
  \end{description}
    Let $X_i^{(1)}$ denote the cube-complex obtained from
    $X_i$ by subdividing each $N$-cube of $X_i$ into $m^N$ isometric
    subcubes; when the subdivision operation is repeated  $k$-times we
    use the notation $X_i^{(k)}$.
\begin{description}
\item[(IOpen)] Each map $\pi_i:X_{i+1}\to X_i^{(1)}$ is open, surjective, cellular, and
    restricts to an isometry on every face.
  \end{description}
A \textbf{gallery} in $X_i^{(k)}$ is a finite sequence of
    $N$-dimensional cells
    \begin{equation}
        \label{eq:ad_sys_0}
      \{\sigma_1,\cdots,\sigma_l\}\subset\ccell
      X_i^{(k)},N.
    \end{equation}
    such that each
    $\sigma_k$ and $\sigma_{k+1}$ share an $(N-1)$-dimensional face.
    If $y\in\sigma_1$ and $y'\in\sigma_l$ we say that the gallery
    connects $y$ to $y'$.
\begin{description}
\item[(IGall)] 
     Any two points in $X_i$ are connected by a gallery.
     For each $x\in X_i^{(1)}$, and for each
    $y,y'\in\pi_i^{-1}(x)$
    there is a gallery (in $X_{i+1}$) of at most $C\gldec$-cells
    joining $y$ to to $y'$;
  \item[(IMeas)] Each $\mu_i$ restricts to a constant multiple (with
    weight $\weight(\mu_i,\sigma)$) of
    Lebesgue measure on each element $\sigma$ of $\ccell X_i,.$ and
    $\pi_{i\#}\mu_{i+1}=\mu_i$. Moreover, there is a uniform constant
    (in $i$)
    $C_\mu$ such that whenever $\sigma,\sigma'\in\ccell X_i,.$ are adjacent:
    \begin{equation}
      \label{eq:ad_sys_1}
      \mu_i(\sigma')\le C_\mu\mu_i(\sigma);
    \end{equation}
  \item[(IPoinc)] Let $f_i\in\ccell X_i^{(1)},N-1.$ and
    $f_{i+1}\in\pi_i^{-1}(f_i)$; then
    the quantity:
    \begin{equation}
      \label{eq:ad_sys_2}
      \sum_{\tau\in\bdcell f_{i+1},.\in\pi_i^{-1}(\sigma_i)}
      \frac{
        \weight(\mu_{i+1},\tau)
      }
      {
        \weight(\mu_i,\sigma_i)
      }
    \end{equation}
    is constant as $\sigma_i$ varies on the set $\bdcell f_{i},.$ of
    $N$-cells of $X_i^{(1)}$ which bound $f_i$.
  \end{description}
\end{defn}
\def\pidec{_\text{\normalfont PI}}
\def\mddec{_\text{\normalfont d}}
For a discussion of the Poincar\'e inequality we refer the reader
to\cite{heinonen98,keith-modulus}.
\begin{defn}[Local PI-space]
  \label{defn:loc_pi}
  A geodesic metric measure space $(X,\mu)$ is a \textbf{local $(1,p)$-PI
  space} if $\mu$ is locally doubling, i.e.~for each ball $B\subset X$
there is a constant $C\mddec(B)$ such that $\mu\on B$ is
$C\mddec(B)$-doubling as a measure on the metric space $B$, and if
there is a constant $C\pidec(B)$ such
  that for each ball $B'\subset B$ and each 
  Lipschitz function $f:X\to\real$ one has:
  \begin{equation}
    \label{eq:loc_pi_1}
    \av_{B'}|f-f_{B'}|\,d\mu \le C\pidec(B)\brad(B')\left(
      \av_{B'}(\biglip f)^p\,d\mu
      \right)^{1/p},
    \end{equation}
where
\begin{equation}
  \label{eq:loc_pi_2}
  \biglip f(x) = \limsup_{y\to x, y\ne x}\frac{|f(y)-f(x)|}{d(x,y)};
\end{equation}
  if $C\pidec(B)$ and $C\mddec(B)$ can be chosen independent of $B$ then $(X,\mu)$ is a
  called a $(1,p)$-PI space.
\end{defn}
This Theorem summarizes properties of AISs.
\begin{thm}[Inverse limits are local $(1,1)$-PI]
  \label{thm:adsys_pi}
  Let $\adsys$ be an admissible inverse system and let
  $\{p_i\}_{i\in I}\subset\prod_{i\in I}X_i$ be a compatible collection
  of basepoints, i.e.~$\pi_i(p_{i+1})=p_i$ $\forall i$. Then the
  following limit (in the pointed measured Gromov-Hausdorff sense) exists
  \begin{equation}
    \label{eq:adsys_pi_s1}
    \begin{aligned}
      \lim_{k\to\infty}(X_k,\mu_k,p_k)&=(X_\infty,\mu_\infty,p_\infty)\\
    \end{aligned}
  \end{equation}
  and is called the \normalfont{\textbf{inverse}} limit of $\adsys$ (given the choice of
  basepoints). Then the inverse limit is a local $(1,1)$-PI space
  where $C\pidec(B)$ and $C\mddec(B)$ depend, besides $B$, only on $C\gmdec$, $m$,
  $C\gldec$ and $C_\mu$. If some
  $(X_k,\mu_k,p_k)$ is a $(1,1)$-PI space, so is the inverse limit.
\end{thm}
\begin{proof}
  \par\noindent\texttt{Step 1: Existence of the inverse limit.}
  \par By \textbf{(IBGeom)}, \textbf{(IOpen)} and \textbf{(IGall)}
  there is a constant $C=C(C\gmdec,C\gldec,m)$ such that:
  \begin{equation}
    \label{eq:adsys_pi_p1}
    \|\pi_i^*d_{X_{i+1}}-d_{X_i}\|\le Cm^{-i}.
  \end{equation}
  By \textbf{(IMeas)} for each $R>0$ there is a constant
  $C=C(R,m,C_\mu)$ such that, if $f:X_i\to[0,1]$ is $1$-Lipschitz with
  $\spt f\subset B(p_i,R)$ then:
  \begin{equation}
    \label{eq:adsys_pi_p2}
    \left|
      \int_{X_i}f\,d\mu_i - \int_{X_i}f\,d(\pi_{i,\#}\mu_{i+1})
    \right| \le Cm^{-i}.
  \end{equation}
  Choosing an appropriate metric to metrize the mGH-topology we conclude
  that the sequence $\{(X_k,\mu_k,p_k)\}_{k\ge\inf I}$ is Cauchy.
  \par\noindent\texttt{Step 2: The Poincar\'e inequality.}
  \par Let $k_0\in I$ and $B\subset X_\infty$ be a ball. As the induced
  map:
  \begin{equation}
    \label{eq:adsys_pi_p3}
    \pi_{\infty,k_0}: X_\infty \to
    X_{k_0}\quad(\text{$\lim_{i\to\infty}$ of $\pi_{i}\circ\pi_{i-1}\circ\cdots\circ\pi_{k_0}$})
  \end{equation}
  is $1$-Lipschitz, there is a finite subcomplex $S_B\subset X_{k_0}$
  whose interior contains $\pi_{\infty,k_0}(B)$; without loss of
  generality we may assume $S_B$ gallery-connected; then $S_B$ is a
  $(1,1)$-PI space since pairs of points in $S_B$ can be joined by
  pencils of curves that satisfy an appropriate modulus estimate
  \cite{keith-modulus}.
  \par The argument in \cite[Sec.~11]{cheeger_inverse_poinc} shows that, if
  $(C_{S_B}, C\mddec(S_B))$ are the constants in the Poincar\'e
  inequality and the doubling condition for
  $(S_B,\mu_{k_0}\on S_B)$, then $(\pi^{-1}_{\infty,k_0}(S_B),
  \mu_\infty\on \pi^{-1}_{\infty,k_0}(S_B))$ is a $(1,1)$-PI space
  with constant in the Poincar\'e inequality:
  \begin{equation}
    \label{eq:adsys_pi_p4}
    C=C(C_{S_B},C\mddec(S_B),m,C\gldec,C_\mu,C\gmdec).
  \end{equation}
  The fact that in~(\ref{eq:loc_pi_1}) one can take the same ball on
  both sides follows from \cite{haj_kosk_sobolev_met} because by
  \textbf{(IGall)} the interior of $S_B$ satisfies an appropriate \emph{chain
  condition}, which passes to $\pi^{-1}_{\infty,k_0}(S_B)$ because of
  \textbf{(IOpen)}.
  \par Finally note that if $(X_{k_0},\mu_{k_0})$ is a $(1,1)$-PI space
  the constants $C_{S_B}$, $C\mddec(S_B)$ can be assumed independent of $S_B$ and so
  $(X_\infty,\mu_\infty)$ is a $(1,1)$-PI space.
\end{proof}
The following definition introduces the kind of systems that we use to
build normal currents. 
\begin{defn}[Weak Admissible Inverse Systems / WAIS]
  \label{defn:wad_sys}
  Let $\adsys$ satisfy \textbf{(IBGeom)}--\textbf{(IMeas)}. We say
  that $\adsys$ is a \textbf{weak admissible inverse system} if the
  following axioms hold:
  \begin{description}
  \item[(IOr)] Each $\sigma\in\ccell X_i,.$ carries an orientation,
    and these orientations induce compatible orientations on the cells
    in $\ccell X_i^{(k)},.$ for each $k\ge 1$. These orientations are
    compatible in the sense that if $\sigma\in\ccell X_{i+1},.$ then
    the orientation of $\pi_i(\sigma)$ is induced by $\pi_i$;
  \item[(IFlux)] There is a $k_0\in I$ such that for $i\ge k_0$ the
    following holds. Fix $f_i\in\ccell ,N-1.$ in the interior of some
    cell in $\ccell X_i,.$ and partition $\bdcell f_i,.$ in two
    subsets $\bdcell ,+.$ and $\bdcell ,-.$ depending on which
    orientation they induce on $f_i$. Then for each $f_i\in\ccell
    ,N-1.$ and each $f_{i+1}\in\pi_i^{-1}(f_i)$ the following holds:
    \begin{equation}
      \label{eq:wad_sys_1}
      \sum_{\tau\in \pi_i^{-1}(\bdcell ,+.)\cap\bdcell f_{i+1},X_i^{(1)}.}
        \weight(\mu_{i+1},\tau) =
      \sum_{\tau\in \pi_i^{-1}(\bdcell ,-.)\cap\bdcell f_{i+1},X_i^{(1)}.}
        \weight(\mu_{i+1},\tau).
    \end{equation}
  \end{description}
\end{defn}
In general \textbf{(IFlux)} is weaker than \textbf{(IPoinc)}.
\def\mkcurr#1.{[\mskip-3mu[#1]\mskip-3mu]}
\begin{defn}[Normal currents associated to a WAIS]
  \label{defn:normal_sys}
  Let $\adsys$ be a WAIS. We can canonically identify each
  $\sigma\in\ccell X_i,.$ with $[0,m^{-i}]^N$ and associate to it a
  (classical) $N$-normal current $\mkcurr\sigma.$ by:
  \begin{equation}
    \label{eq:normal_sys_1}
    \mkcurr\sigma. =
    \pm \partial_1\wedge\cdots\wedge\partial_N\,\lebmeas N.\on\sigma,
  \end{equation}
  where the choice of $\pm$ depends on the choice of orientation on
  $\sigma$. To each $X_i$ we can associate a (metric) normal current $N_i$
  (where $\cmass N_i.$ and $\cmass\partial N_i.$ are locally finite)
  by:
  \begin{equation}
    \label{eq:normal_sys_2}
    N_i=\sum_{\sigma\in\ccell X_i,.}\weight(\mu_i,\sigma)\mkcurr\sigma.;
  \end{equation}
  \textbf{(IOr)} guarantees that:
  \begin{equation}
    \label{eq:normal_sys_3}
    \pi_{i\#}N_{i+1}=N_i.
  \end{equation}
\end{defn}
\begin{thm}[Limit normal currents for WAIS]
  \label{thm:limit_curr_wadsys}
  Let $\adsys$ be a WAIS and let
  $\{p_i\}_{i\in I}\subset\prod_{i\in I}X_i$ be a compatible collection
  of basepoints. Then:
  \begin{description}
  \item[(mGH)] The following limit exists as in
    Theorem~\ref{thm:adsys_pi}:
    \begin{equation}
      \label{eq:limit_curr_wadsys_s1}
      \begin{aligned}
        \lim_{k\to\infty}(X_k,\mu_k,p_k)&=(X_\infty,\mu_\infty,p_\infty);
      \end{aligned}
    \end{equation}
  \item[(Nag)] For $i\in I\cup\{\infty\}$ the metric
    space $X_i$ has small Nagata dimension $N$ with uniform
    parameters (in $i$: scale and cover-separation). If some $X_{k_0}$ has Nagata dimension $N$, so do all
    the $X_i$ with uniform cover-separation parameter;
  \item[(Wea)] For $i\in I\cup\{\infty\}$ the
    module $\wder\mu_i.$ is free on $N$-generators
    $\{D_{i,\alpha}\}_{\alpha\in\{1,\cdots,N\}}$. If $i\in \zahlen$
    then for each $\sigma\in\ccell X_i,.$:
    \begin{equation}
      \label{eq:limit_curr_wadsys_s2}
      \chi_{\sigma}D_{i,\alpha}=\partial_\alpha,
    \end{equation}
    $\partial_\alpha$ being the ``Euclidean derivation'' in the
    $\alpha$-direction in the cell $\sigma$. The derivations are
    compatible in the sense that whenever $i\ge j$:
    \begin{equation}
      \label{eq:limit_curr_wadsys_s3}
      \pi_{i,j\#}D_{i,\alpha}=D_{j,\alpha};
    \end{equation}
  \item[(Nor)] Assume that convergence
    in~(\ref{eq:limit_curr_wadsys_s1}) takes place (as weak*
    convergence of measures and standard Hausdorff-Vietoris
    convergence) in some complete separable metric space $Z$ and that
    there is a constant $C_Z$ such that if $i,j\in I\cup\{\infty\}$ with $i\ge
    j$ one has:
    \begin{equation}
      \label{eq:limit_curr_wadsys_s3bis}
      \sup_{x_i\in X_i} d_Z(x_i,\pi_{i,j}(x_i))\le C_Zm^{-j}.
    \end{equation}
    Then the
    following limit of currents exists in the weak topology for normal
    currents:
    \begin{equation}
      \label{eq:limit_curr_wadsys_s4}
      \begin{aligned}
        \lim_{k\to\infty} N_k&=N_\infty,
      \end{aligned}
    \end{equation}
    and 
    \begin{equation}
      \label{eq:limit_curr_wadsys_s5}
      \pi_{i,j\#}N_i=N_j;
    \end{equation}
    furthermore, each current $N_i$ is \normalfont{\textbf{simple}} in the sense
    that:
    \begin{equation}
      \label{eq:limit_curr_wadsys_s6}
      N_i = D_{i,1}\wedge\cdots\wedge D_{i,N}\,\mu_i.
    \end{equation}
    Finally, the convergence in~(\ref{eq:limit_curr_wadsys_s4}) does
    not entail loss of mass: i.e.~for each open $U\Subset X_j$:
    \begin{equation}
      \label{eq:limit_curr_wadsys_s7}
      \begin{aligned}
        \cmass N_i.(\pi_{i,j}^{-1}(U))&=\cmass N_j.(U)\\
        \cmass \partial N_i.(\pi_{i,j}^{-1}(U))&=\cmass \partial N_j.(U).
      \end{aligned}
    \end{equation}
  \end{description}
\end{thm}
\begin{rem}[Assumption~(\ref{eq:limit_curr_wadsys_s3bis})]
  \label{rem:pi_uf_assmp}
  Note that assumption~(\ref{eq:limit_curr_wadsys_s3bis}) is not
  restrictive as by induction and passing to
  the limit in~(\ref{eq:adsys_pi_p1}) one has:
  \begin{equation}
    \label{eq:pi_uf_assmp_1}
    \|\pi_{i,j}^*d_{X_{j}}-d_{X_i}\|\le Cm^{-j};
  \end{equation}
  one can then find a metric on $Z_i=X_\infty\sqcup X_i$ extending
  $d_{X_\infty}$ and $d_{X_i}$ such
  that~(\ref{eq:limit_curr_wadsys_s3bis}) holds for the pair of indices
  $(\infty,i)$. Then one glues the spaces $Z_i$ across $X_\infty$ and uses the
  surjectivity of the maps $\{\pi_{i,j}\}$ to deduce~(\ref{eq:limit_curr_wadsys_s3bis}).  
\end{rem}
\def\mxgall#1.{\setbox1=\hbox{$#1$\unskip} 
  		 \text{\normalfont Mx}
		 ({\ifdim\wd1>0pt #1\else X_{i+1}\fi})}
\pcreatenrm{wd}{\wder\mu_\infty.}{|}
\def\tccell#1,#2.{\setbox1=\hbox{$#1$\unskip}\setbox2=\hbox{$#2$\unskip} 
  		 \widetilde{\text{\normalfont Cell}}_{\ifdim\wd2>0pt #2\else N\fi}
		 ({\ifdim\wd1>0pt #1\else X_i^{(1)}\fi})}
\begin{proof}[Proof of Theorem~\ref{thm:limit_curr_wadsys}]
  \par\noindent\texttt{Step 1:  Proof of {\normalfont\textbf{(mGH)}}
    and {\normalfont\textbf{(Nag)}}}.
  \par The existence of the inverse limit follows as in
  Theorem~\ref{thm:adsys_pi}. The proof of \textbf{(Nag)} follows from
  Theorem~\ref{thm:tower}; in fact, if $\sigma\in\ccell
  X_{k_0},.$ we have that the maps:
  \begin{equation}
    \label{eq:limit_currs_wadsys_p1}
    \pi_{\infty,i}:\pi^{-1}_{\infty,k_0}(\sigma)\to\pi_i(\pi^{-1}_{\infty,k_0}(\sigma))
  \end{equation}
  show that $\pi^{-1}_{\infty,k_0}(\sigma)$ has property
  \textbf{(TPA)}($N$) and so it has small Nagata dimension at most $N$. A
  lower bound on the small Nagata dimension follows from a lower bound
  on the topological dimension as the map
  in~(\ref{eq:limit_currs_wadsys_p1}) is light and so
  cannot decrease the topological dimension (\cite[Thm.~1.24.4]{engelking_topbook}).
  \par The claim
  about the cover-separation and scale being uniform holds as all elements of $\sigma\in\ccell
  X_{k_0},.$ are isometric to a rescaled copy of $[0,1]^N$. The claim
  about the large Nagata dimension follows because in that case one can
  replace $\sigma$ by $X_{k_0}$ (and the Nagata covers for
  $X_{k_0}$ can be also used to produce Nagata covers for $X_{j}$
  if $j<k_0$ at scales $>m^{-j}$).
  \par\noindent\texttt{Step 2: Definition of maximal galleries.}
  \par We start the proof of \textbf{(Wea)} by constructing the
  derivations $D_{i,\alpha}$; without loss of generality we take
  $\alpha=1$ and by locality assume that $\inf I=0$, $X_0=[0,1]^N$ and
  $\mu_0=\lebmeas N.\on[0,1]^N$. As the coordinate functions $x_\alpha$
  on $[0,1]^N$ can be canonically pulled back to each $X_i$ we will just
  write $x_\alpha$ for $x_\alpha\circ\pi_{i,j}$ in the following.
  \par A string of cells:
  \begin{equation}
    \label{eq:limit_currs_wadsys_p2}
    \{\sigma_0,\cdots,\sigma_t\}\subset\ccell X_i^{(k)},.
  \end{equation}
  is an \textbf{$x_1$-gallery} if:
  \begin{description}
  \item[(Mx1)] $\sigma_i\ne\sigma_{i+1}$ and $\max x_1(\sigma_i) = \min
    x_1(\sigma_{i+1})$ for $0\le i\le t-1$;
  \item[(Mx2)] $\sigma_i$ and $\sigma_{i+1}$ share a codimension-$1$
    face on which $x_1$ is constant. 
  \end{description}
  If an $x_1$-gallery $g$ contains an
  $x_1$-gallery $g'$ we say that $g$ \textbf{extends} $g'$; if $g$
  does not admit an $x_1$-gallery properly extending it, then $g$ is
  called \textbf{maximal} and the set of maximal $x_1$-galleries is
  denoted by $\mxgall X_i^{(k)}.$. Given $\sigma\in\ccell X_i^{(k)},.$ we use $\partial_+\sigma$ to
  denote the union of the cells of $\ccell X_{i}^{(k)},.$ which bound the
  $(N-1)$-dimensional face of $\sigma$ on which $x_1$ is maximal and
  induce on it an orientation opposite to that induced by $\sigma$
  ($\partial_-\sigma$ is defined similarly considering the face on which
  $x_1$ is minimal).
  \par To $g\in\mxgall X_i^{(k)}.$ we can associate the measure
  $\lebmeas N.\on g$ which is just Lebesgue measure on each cell of
  $g$. We now construct measures $Q_i$ on $\mxgall X_i.$ such that:
  \begin{equation}
    \label{eq:limit_currs_wadsys_p3}
    \mu_i=\sum_{g\in\mxgall X_i.}\lebmeas N.\on g\,Q_i(g).
  \end{equation}
  Note that each $g\in\mxgall X_i^{(k)}.$ is contained in a unique
  $\hat{g}\in \mxgall X_i.$ and if we let $Q_i^{(k)}(g)=Q_i(\hat{g})$,
  (\ref{eq:limit_currs_wadsys_p3}) implies the more general version:
  \begin{equation}
    \label{eq:limit_currs_wadsys_p4}
    \mu_i=\sum_{g\in\mxgall X_i^{(k)}.}\lebmeas N.\on g\,Q_i^{(k)}(g).
  \end{equation}
  As $\mxgall X_0.$ is a singleton, we let $Q_0=1$ on $\mxgall X_0.$ so
  that~(\ref{eq:limit_currs_wadsys_p3}) holds for $i=0$. The measure
  $Q_{i+1}$ is defined by recursion: given
  \begin{equation}
    \label{eq:limit_currs_wadsys_p5}
    g=\{\sigma_1,\cdots,\sigma_{m^{i+1}}\}\in\mxgall X_{i+1}.
  \end{equation}
  we let:
  \begin{equation}
    \label{eq:limit_currs_wadsys_p6}
    Q_{i+1}(g)=Q_i^{(1)}(\pi_i(g))\times\frac{\mu_{i+1}(\sigma_1)}{\mu_i(\pi_i(\sigma_1))}
    \frac{\mu_{i+1}(\sigma_2)}{\mu_{i+1}(\partial_+\sigma_1)}\times\cdots\times\frac{\mu_{i+1}(\sigma_{m^{i+1}})}{\mu_{i+1}(\partial_+\sigma_{m^{i+1}-1})}.
  \end{equation}
  \par\noindent\texttt{Step 3: Proof
    of~{\normalfont~(\ref{eq:limit_currs_wadsys_p3})}.}
  \par We prove~(\ref{eq:limit_currs_wadsys_p3}) by induction; assume
  that it holds for $\mu_i$ and let
  \begin{equation}
    \label{eq:limit_currs_wadsys_p7}
    \nu_{i+1}=\sum_{g\in\mxgall.}\lebmeas N.\on g\,Q_{i+1}(g);
  \end{equation}
  it suffices to show that whenever $\sigma\in\ccell X_{i+1},.$ one has:
  \begin{equation}
    \label{eq:limit_currs_wadsys_p8}
    \nu_{i+1}(\sigma)=\mu_{i+1}(\sigma).
  \end{equation}
  \par To $\sigma$ it is associated a unique $t\in\{1,\cdots,m^{i+1}\}$
  such that, whenever $g\in\mxgall X_{i+1}.$ extends $\sigma$, $\sigma$
  is the $t$-th element of $g$. In~(\ref{eq:limit_currs_wadsys_p6}) if
  we sum on $\{\sigma_{t+1},\cdots,\sigma_{m^{i+1}}\}$ and use
  \textbf{(IFlux)} we get:
  \begin{equation}
    \label{eq:limit_currs_wadsys_p9}
    \begin{split}
      Q_{i+1}(\{\text{$g$ extends
        $\{\sigma_1,\cdots,\sigma_t\}$})&=Q_i^{(1)}(\{\text{$\pi_i(g)$
        extends $\{\pi_i(\sigma_1),\cdots,\pi_i(\sigma_t)\}$}))\\
      &\mskip \munsplit\times\frac{\mu_{i+1}(\sigma_1)}{\mu_i(\pi_i(\sigma_1))}
      \frac{\mu_{i+1}(\sigma_2)}{\mu_{i+1}(\partial_+\sigma_1)}\times\cdots\times\frac{\mu_{i+1}(\sigma_{t})}{\mu_{i+1}(\partial_+\sigma_{t-1})}.
    \end{split}
  \end{equation}
  We will compute  $Q_{i+1}(\{\text{$g$ extends
    $\{\sigma_t\}$}\})$ by summing~(\ref{eq:limit_currs_wadsys_p9})
  on $\sigma_1,\cdots,\sigma_{t-1}$; note that the order of
  summation matters as $\sigma_t$ is kept fixed. Concretely,
  $\sigma_{t-1}\in\partial_{-}\sigma_t$,
  $\sigma_{t-2}\in\partial_{-}\sigma_{t-1}$, etc\dots We thus start
  by removing the innermost sums, i.e.~start with $\sigma_1$, then
  $\sigma_2$, etc\dots\ For example, using \textbf{(IFlux)} and \textbf{(IOpen)}:
  \begin{equation}
    \label{eq:limit_currs_wadsys_p10}
    \begin{split}
      \sum_{\sigma_1}Q_{i+1}(\{
      \text{$g$ extends $\{\sigma_1,\cdots,\sigma_t\}$}
      \})
      &= \sum_{\pi_i(\sigma_1)}Q_i^{(1)}(\{\text{$\pi_i(g)$
        extends $\{\pi_i(\sigma_1),\cdots,\pi_i(\sigma_t)\}$}\})\\
      &\mskip \munsplit\times\frac{\mu_{i+1}(\sigma_2)}{\mu_i(\pi_i(\sigma_1))}
      \frac{\mu_{i+1}(\sigma_3)}{\mu_{i+1}(\partial_+\sigma_2)}\times\cdots\times\frac{\mu_{i+1}(\sigma_{t})}{\mu_{i+1}(\partial_+\sigma_{t-1})}.
    \end{split}
  \end{equation}
  But note that for $2\le j\le m$, $\pi_i(\sigma_j)$ and
  $\pi_i(\sigma_{j-1})$ belong to the same cell of $\ccell X_i,.$
  implying
  \begin{equation}
    \label{eq:limit_currs_wadsys_p11}
    \mu_i(\pi_i(\sigma_j))=\mu_i(\pi_i(\sigma_{j-1})),
  \end{equation}
  from which we obtain:
  \begin{equation}
    \label{eq:limit_currs_wadsys_p12}
    \begin{split}
      Q_{i+1}(\{\text{$g$ extends
        $\{\sigma_2,\cdots,\sigma_t\}$}\})&=Q_i^{(1)}(\{\text{$\pi_i(g)$
        extends $\{\pi_i(\sigma_2),\cdots,\pi_i(\sigma_t)\}$}\})\\
      &\mskip \munsplit\times\frac{\mu_{i+1}(\sigma_2)}{\mu_i(\pi_i(\sigma_2))}
      \frac{\mu_{i+1}(\sigma_3)}{\mu_{i+1}(\partial_+\sigma_2)}\times\cdots\times\frac{\mu_{i+1}(\sigma_{t})}{\mu_{i+1}(\partial_+\sigma_{t-1})}.
    \end{split}
  \end{equation}
  We can iterate the previous argument up to $m$; in particular, if
  $t\le m$ we would get:
  \begin{equation}
    \label{eq:limit_currs_wadsys_p13}
    Q_{i+1}(\{\text{$g$ extends $\{\sigma_t\}$}\}) =
    \frac{Q_i^{(1)}(\{\text{$\pi_i(g)$ extends $\{\pi_i(\sigma_t)\}$}\})}
    {\mu_i(\pi_i(\sigma_t))}\mu_{i+1}(\sigma_t).
  \end{equation}
  We want to generalize~(\ref{eq:limit_currs_wadsys_p13}) also for $t>m$
  and the main point is illustrated in passing from $m$ to $m+1$. We
  start with:
  \begin{equation}
    \label{eq:limit_currs_wadsys_p14}
    \begin{split}
      Q_{i+1}(\{\text{$g$ extends
        $\{\sigma_m,\cdots,\sigma_t\}$}\})&=\frac{Q_i^{(1)}(\{\text{$\pi_i(g)$
          extends $\{\pi_i(\sigma_m),\cdots,\pi_i(\sigma_t)\}$}\})}{\mu_i(\pi_i(\sigma_m))}\\
      &\mskip \munsplit\times\mu_{i+1}(\sigma_m)\frac{\mu_{i+1}(\sigma_{m+1})}{\mu_{i+1}(\partial_+\sigma_{m})}
      \times\cdots\times\frac{\mu_{i+1}(\sigma_{t})}{\mu_{i+1}(\partial_+\sigma_{t-1})}.
    \end{split}
  \end{equation}
  We now use the definition of $Q_i$ to deduce:
  \begin{equation}
    \label{eq:limit_currs_wadsys_p15}
    \begin{split}    
      Q_i^{(1)}(\{
      \text{$g$ extends $\{\tau_m,\cdots,\tau_t\}$}\}
      )
      &= \frac{
        Q_i^{(2)}(\{
        \text{$\pi_{i-1}(g)$ extends $\{\pi_{i-1}(\tau_m),\cdots,\pi_{i-1}(\tau_t)\}$}\}
        )
      } {
        \mu_{i-1}(\pi_{i-1}(\tau_m))
      }
      \\
      &\mskip \munsplit\times\mu_i(\tau_m)
      \frac{\mu_{i}(\tau_{m+1})}{\mu_{i}(\partial_+\tau_{m})}
      \times\cdots\times\frac{\mu_{i}(\tau_{t})}{\mu_{i}(\partial_+\tau_{t-1})}.
    \end{split}
  \end{equation}
  Now $\pi_{i-1}(\tau_m)$ and $\pi_{i-1}(\tau_{m+1})$ belong to the same
  cell of $\ccell X_{i-2},.$; we thus have:
  \begin{equation}
    \label{eq:limit_currs_wadsys_p16}
    \mu_i(\pi_{i-1}(\tau_m))=\mu_i(\pi_{i-1}(\tau_{m+1})),
  \end{equation}
  and we can sum~(\ref{eq:limit_currs_wadsys_p15}) on $\tau_m$;
applying \textbf{(IFlux)} and \textbf{(IOpen)}, and we thus obtain:
  \begin{equation}
    \label{eq:limit_currs_wadsys_p17}
    \begin{split}
      \sum_{\tau_m}
      Q_i^{(1)}(\{
      \text{$g$ extends $\{\tau_m,\cdots,\tau_t\}$}\}
      )
      &= \frac{
        Q_i^{(2)}(\{
        \text{$\pi_{i-1}(g)$ extends $\{\pi_{i-1}(\tau_{m+1}),\cdots,\pi_{i-1}(\tau_t)\}$}\}
        )
      } {
        \mu_{i-1}(\pi_{i-1}(\tau_{m+1}))
      }
      \\
      &\mskip \munsplit\times\mu_i(\tau_{m+1})
      \frac{\mu_{i}(\tau_{m+2})}{\mu_{i}(\partial_+\tau_{m+1}))}
      \times\cdots\times\frac{\mu_{i}(\tau_{t})}{\mu_{i}(\partial_+\tau_{t-1})}\\
      &=
      Q_i^{(1)}(\{
      \text{$g$ extends $\{\tau_{m+1},\cdots,\tau_t\}$}\}
      );
    \end{split}
  \end{equation}
  thus if in~(\ref{eq:limit_currs_wadsys_p14}) we sum over $\sigma_m$ we
  obtain:
  \begin{equation}
    \label{eq:limit_currs_wadsys_p18}
    \begin{split}
      Q_{i+1}(\{\text{$g$ extends
        $\{\sigma_{m+1},\cdots,\sigma_t\}$}\})&=\frac{Q_i^{(1)}(\{\text{$\pi_i(g)$
          extends $\{\pi_i(\sigma_{m+1}),\cdots,\pi_i(\sigma_t)\}$}\})}{\mu_i(\pi_i(\sigma_{m+1}))}\\
      &\mskip \munsplit\times\mu_{i+1}(\sigma_{m+1})\frac{\mu_{i+1}(\sigma_{m+2})}{\mu_{i+1}(\partial_+\sigma_{m+1}))}
      \times\cdots\times\frac{\mu_{i+1}(\sigma_{t})}{\mu_{i+1}(\partial_+\sigma_{t-1})}.
    \end{split}
  \end{equation}
  Continuing by induction we
  establish~(\ref{eq:limit_currs_wadsys_p13}). 
  \par Now from the inductive hypothesis $\nu_i=\mu_i$ we have:
  \begin{equation}
    \label{eq:limit_currs_wadsys_p19}
    Q_{i}^{(1)}(\{\text{$\pi_i(g)$ extends
      $\pi_i(\sigma_t)$}\}) = m^{(i+1)N}\mu_i(\pi_i(\sigma_t)),
  \end{equation}
  from which (using~(\ref{eq:limit_currs_wadsys_p13})) we get:
  \begin{equation}
    \label{eq:limit_currs_wadsys_p20}
    Q_{i+1}(\{\text{$g$ extends
      $\sigma_t$}\}) = m^{(i+1)N}\mu_{i+1}(\sigma_t)
  \end{equation}
  which implies~(\ref{eq:limit_currs_wadsys_p8}) for $\sigma=\sigma_t$.
  \par\noindent\texttt{Step 4: Construction of Weaver derivations.}
  \par Note that $\pi_{i\#}Q_{i+1}=Q_i^{(1)}$ by the recursive
  definition of $Q_{i+1}$. Let $g\in\mxgall X_i^{(k)}.$ and use $g(0)$ to
  denote the first cell of $g$; then any $y$ on the face $f_{-}$ of
  $g(0)$ where $x_1$ is minimal uniquely determines a unit-speed geodesic segment 
  \begin{equation}
    \label{eq:limit_currs_wadsys_p21}
    \gamma:[0,1]\to X_i^{(k)}
  \end{equation}
  such that:
  \begin{align}
    \label{eq:limit_currs_wadsys_p22}
    x_1\circ\gamma(t)&=t\\
    \label{eq:limit_currs_wadsys_p23}
    \im \gamma&\subset g;
  \end{align}
  if we choose $y\in f_{-}$ according to $\lebmeas N-1.\on f_{-}$ and
  use the Fubini representation in the $x_1$-direction on each cell of
  $g$, (\ref{eq:limit_currs_wadsys_p3}) implies the representation
  \begin{equation}
    \label{eq:limit_currs_wadsys_p24}
    \mu_i = \int\hmeas._\gamma\,dP_i^{(k)}(\gamma)\quad(i\in I)
  \end{equation}
  where $P_i^{(k)}$ is a probability measure concentrated on the set $\Omega_i$ of geodesic segments
  satisfying~(\ref{eq:limit_currs_wadsys_p21})
  and~(\ref{eq:limit_currs_wadsys_p22}). 
  Note also that cellular subdivision
  does not affect the probabilities, i.e.~$P_i^{(k)}=P_i$ and that
  $\pi_{i\#}Q_{i+1}=Q_i^{(1)}$ implies
  \begin{equation}
    \label{eq:limit_currs_wadsys_p25}
    \pi_{i\#}P_{i+1}=P_i^{(1)}=P_i.
  \end{equation}
  \par Having arranged mGH-convergence in some (proper and complete) $Z$ we see
  that the measures $\{P_i\}_{i\in I}$ form a Cauchy sequence in the
  weak* topology, and passing to a limit we obtain:
  \begin{equation}
    \label{eq:limit_currs_wadsys_p26}
    \mu_{\infty}=\int\hmeas._\gamma\,dP_\infty(\gamma)
  \end{equation}
  where $P_\infty$ is concentrated on $\Omega_\infty$. We finally let
  $D_{i,1}$ be the Weaver derivation associated to the Alberti
  representation $[P_i,1]$. To construct $D_{i,\alpha}$ one uses
  \texttt{Step 3} and this step for $x_\alpha$.
  \par\noindent\texttt{Step 5:  Proof of {\normalfont\textbf{(Wea)}}.}
  \par Equation~(\ref{eq:limit_curr_wadsys_s2}) follows because the
  Alberti representations used to define $D_{i,\alpha}$ restrict to
  appropriate Fubini-like representations in the $x_\alpha$ direction on
  each cell. Equation~(\ref{eq:limit_curr_wadsys_s3}) follows from the
  compatibility~(\ref{eq:limit_currs_wadsys_p25}) between the $P_i$ (the
  case where $i=\infty$ being handled by a limiting argument).
  \par We now let $i\in I\cup\{\infty\}$; we have:
  \begin{equation}
    \label{eq:limit_currs_wadsys_p27}
    D_{i,\alpha}x_\beta = \delta_{\alpha,\beta}
  \end{equation}
  so that the index of $\wder\mu_i.$ is at least $N$. On the other hand
  by \texttt{Step 1} $X_i$ has property \textbf{(TAP)}$(N)$ and so
  $\wder\mu_i.$ must have index $N$ and be free on $N$ generators. If
  $i$ is finite the fact that $\{D_{i,\alpha}\}_\alpha$ gives a basis of $\wder\mu_i.$
  follows immediately from~(\ref{eq:limit_curr_wadsys_s2}). To show that
  $\{D_{\infty,\alpha}\}$ is a basis for $\wder\mu_\infty.$ we choose a
  basis $\{\tilde{D}_{\infty,\alpha}\}$ for $\wder\mu_\infty.$
  satisfying:
  \begin{equation}
    \label{eq:limit_currs_wadsys_p28}
    \wdnrm\tilde{D}_{\infty,\alpha}.=1\quad(\forall\alpha),
  \end{equation}
  and then choose a matrix-valued measurable map $(M_{\alpha,\beta})$ with:
  \begin{equation}
    \label{eq:limit_currs_wadsys_p29}
    D_{\infty,\alpha}=\sum_\beta M_{\alpha,\beta}\tilde{D}_{\infty,\beta};
  \end{equation}
  using the functions $x_\gamma$ we arrive at the identity:
  \begin{equation}
    \label{eq:limit_currs_wadsys_p30}
    \text{Id}_N = (D_{\infty,\alpha}x_\gamma)=(M_{\alpha,\beta})\cdot (\tilde{D}_{\infty,\beta}x_\gamma),
  \end{equation}
  which provides a lower bound on $\det(M_{\alpha,\beta})$ showing that 
  $\{D_{\infty,\alpha}\}$ is a basis for $\wder\mu_\infty.$.
  \par\noindent\texttt{Step 6:  Proof
    of~{\normalfont~(\ref{eq:limit_curr_wadsys_s7})} when {\normalfont
      $i,j$} are finite.}
  \par Note that for $i=j+1$ (\ref{eq:limit_curr_wadsys_s5}) is
  just~(\ref{eq:normal_sys_3}). In particular, as long as $i$ is finite,
  (\ref{eq:limit_curr_wadsys_s6}) follows from the definition of $N_i$,
  (\ref{eq:normal_sys_2}), and the construction of the Weaver
  derivations $\{D_{i,\alpha}\}$ as long as one chooses the ``order'' of
  the coordinates $x_\alpha$ on each cell of $X_{\inf I}$ compatibly
  with the orientations. Thus, we have:
  \begin{equation}
    \label{eq:limit_currs_wadsys_p31}
    \cmass N_i.=\mu_i,
  \end{equation}
  and so the first equation in~(\ref{eq:limit_curr_wadsys_s7}) follows.
  \par Taking the boundary in~(\ref{eq:normal_sys_2}) we get:
  \begin{equation}
    \label{eq:limit_currs_wadsys_p32}
    \begin{split}
      \partial N_i &= \sum_{\sigma\in\ccell
        X_i,.}\weight(\mu_i,\sigma)\,\partial\mkcurr\sigma.
      =\sum_{\tau\in\ccell
        X_{i-1}^{(1)},.}\sum_{\sigma\in\pi_{i-1}^{-1}(\tau)}
      \weight(\mu_i,\sigma)\,\partial\mkcurr\sigma.\\
      &=\sum_{f_{i-1}\in\ccell
        X_{i-1}^{(1)},.}\sum_{f_i\in\pi_{i-1}^{-1}(f_{i-1})}
      \biggl\{
      \sum_{\sigma\in \pi_{i-1}^{-1}(\bdcell f_{i-1},+.)\cap\bdcell f_i,X_i.}
      \weight(\mu_i,\sigma)\\
      &\mskip\munsplit - \sum_{\sigma\in \pi_{i-1}^{-1}(\bdcell
        f_{i-1},-.)\cap\bdcell f_i,X_i.}
      \weight(\mu_i,\sigma)
      \biggr\}\mkcurr f_i..
    \end{split}
  \end{equation}
  Note that in~(\ref{eq:limit_currs_wadsys_p32}) the terms which
  correspond to $f_{i-1}$ lying in the interior of some $\tau\in\ccell
  X_{i-1},.$ vanish because of \textbf{(IFlux)}. Let $U\Subset X_{i-1}$ open;
  then
  \begin{equation}
    \label{eq:limit_currs_wadsys_p33}
    \begin{split}
      \cmass\partial N_{i-1}.(U)&=\sum_{f_{i-1}\in\ccell X_{i-1},N-1.}
      \left|
        \sum_{\tau\in\bdcell f_{i-1},+.}\weight(\mu_{i-1},\tau)-
        \sum_{\tau\in\bdcell f_{i-1},-.}\weight(\mu_{i-1},\tau)
      \right|\\
      &\times\mskip\munsplit\hmeas N-1.(f_{i-1}\cap U);
    \end{split}
  \end{equation}
let $\tccell X_{i-1}^{(1)},N-1.$ denote the set of those
$f_{i-1}\in\ccell X_{i-1}^{(1)},N-1.$ on the boundary of some cell of
$\ccell X_{i-1},.$; then using \textbf{(IMeas)}, \textbf{(IOr)} and
(\ref{eq:limit_currs_wadsys_p32}) we obtain:
\begin{equation}
  \label{eq:limit_currs_wadsys_p34}
  \begin{split}
    \cmass\partial N_{i-1}.(U)&=\sum_{f_{i-1}\in\tccell X_{i-1}^{(1)},N-1.}
      \left|
        \sum_{\tau\in\bdcell f_{i-1},+.}\weight(\mu_{i-1},\tau)-
        \sum_{\tau\in\bdcell f_{i-1},-.}\weight(\mu_{i-1},\tau)
      \right|\\
      &\times\mskip\munsplit\hmeas N-1.(f_{i-1}\cap U)\\
      &=\sum_{f_{i-1}\in\tccell X_{i-1}^{(1)},N-1.}
      \sum_{f_i\in\pi_{i-1}^{-1}(f_{i-1})}
      \biggl|
        \sum_{\sigma\in \pi_{i-1}^{-1}(\bdcell f_{i-1},+.)\cap\bdcell f_i,X_i.}
        \weight(\mu_i,\sigma)\\
        &\mskip\munsplit -
        \sum_{\sigma\in \pi_{i-1}^{-1}(\bdcell f_{i-1},-.)\cap\bdcell f_i,X_i.}
        \weight(\mu_i,\sigma)
      \biggr|\\
      &\times\mskip\munsplit\hmeas N-1.(f_i\cap\pi_{i-1}^{-1}(U))\\
      &=\cmass\partial N_i.(\pi_{i-1}^{-1}(U)),
  \end{split}
\end{equation}
from which we conclude that the second equation
in~(\ref{eq:limit_curr_wadsys_s7}) also holds for $i,j$ finite.
\par\noindent\texttt{Step 7:  Weak convergence for normal currents.}
\par For a discussion about weak convergence for normal currents we
refer the reader to \cite[Sec.~5]{lang_local_currents}. The main point is that
on sets of normal currents for which one has uniform bounds on the
masses and the masses of the boundaries (e.g.~the situation in
\texttt{Step 6}) the weak topology is metrizable. Concretely, consider
the following set of ``formal'' $N$-forms: 
\begin{multline}
  \label{eq:limit_curr_wadsys_p35} 
  \Omega=\{
  \text{$\omega=f_0\,df_1\wedge\cdots\wedge df_N$ where
    $f_\alpha:Z\to\real$ are $1$-Lipschitz,}\\ 
  \text{$\spt f_\alpha$ is
    bounded and $|f_0|\le 1$}
  \};
\end{multline}
to show that $\{N_i\}_{i\in I}$ converges it suffices to show that, for
each $\omega\in\Omega$, $\{N_i(\omega)\}_{i\in I}$ is a Cauchy
sequence. On $\Omega$ we introduce (pseudo)distances:
\begin{equation}
  \label{eq:limit_curr_wadsys_p36}
  \begin{aligned}
    \|\omega-\omega'\|_{\Omega}&=\max_{\alpha=0,\cdots,N}\|f_\alpha-f'_\alpha\|_{\infty}\\
    \|\omega-\omega'\|_{\Omega,
      Y}&=\max_{\alpha=0,\cdots,N}\sup_{y\in
      Y}|(f_\alpha-f'_\alpha)(y)|\quad(Y\subset X),
  \end{aligned}
\end{equation}
and finally define the ``support'' $\spt\omega$ of $\omega$ to be
$\spt f_0$.
\par We will use  \cite[Thm.~5.2]{lang_local_currents}: there is a $C(N)$
such that if $T$ is an $N$-dimensional normal current:
\begin{equation}
  \label{eq:limit_curr_wadsys_p37}
  |T(\omega)-T(\omega')|\le C(N)\|\omega-\omega'\|_{\Omega,\spt T}(\cmass
  T.(\spt\omega) +\cmass\partial T.(\spt\omega)),
\end{equation}
whenever $\omega,\omega'\in\Omega$.
Fix $i,j\in I\cup\{\infty\}$ and $\omega\in\Omega$. We let (which is
only well-defined on $X_i$) $\pi_{i,j}^*\omega$ denote the pull-back
of $\omega$ (restricted on $X_j$), i.e.:
\begin{equation}
  \label{eq:limit_curr_wadsys_p38}
  \pi_{i,j}^*\omega=f_0\circ\pi_{i,j}\,df_1\circ\pi_{i,j}\wedge\cdots\wedge
  df_N\circ\pi_{i,j};
\end{equation}
by~(\ref{eq:limit_curr_wadsys_s3bis}) we have:
\begin{equation}
  \label{eq:limit_curr_wadsys_p39}
  \|\omega-\pi_{i,j}^*\omega\|_{\Omega, X_i}=O(m^{-j})
\end{equation}
where the constants hidden in the $O(\cdot)$-notation are uniform in
$i,j$. 
As
\begin{equation}
  \label{eq:limit_curr_wadsys_p40}
  |N_{i+1}(\omega)-N_i(\omega)|=|N_{i+1}(\omega)-N_{i+1}(\pi_{i}^*\omega)|,
\end{equation}
combining~(\ref{eq:limit_curr_wadsys_p37}),
(\ref{eq:limit_curr_wadsys_p39}) with the uniform bounds on $\cmass
N_i.(\spt\omega)$ and $\cmass \partial N_i.(\spt\omega)$ given by
\texttt{Step 6}, we conclude that $\{N_i(\omega)\}_{i\in I}$ is Cauchy
and thus the limit $N_\infty$ exists.
\par\noindent\texttt{Step 8: Proof of
  {\normalfont~(\ref{eq:limit_curr_wadsys_s5})},
  {\normalfont~(\ref{eq:limit_curr_wadsys_s6})}
  and {\normalfont~(\ref{eq:limit_curr_wadsys_s7})} when $i=\infty$.}
\par To show~(\ref{eq:limit_curr_wadsys_s5}) we use that $N_\infty$ is
normal (i.e.~the joint continuity~(\ref{eq:boundary_7})) to deduce:
\begin{equation}
  \label{eq:limit_curr_wadsys_p41}
  N_\infty (\pi_{\infty,j}^*\omega)=\lim_{k\to\infty}N_{\infty}(\pi_{k,j}^*\omega),
\end{equation}
where for each $k$ $\pi_{k,j}^*\omega$ has been extended from $X_k$ to
$Z$ using MacShane's Lemma on each $f_\alpha$; we then use the
definition of the weak topology and~(\ref{eq:limit_curr_wadsys_p39}):
\begin{equation}
  \label{eq:limit_curr_wadsys_p42}
  \begin{split}
    N_\infty(\pi_{\infty,j}^*\omega)&=\lim_{k\to\infty}\sup_{i\ge
      k}N_i(\pi_{k,j}^*\omega)\\
    &=\lim_{k\to\infty}\sup_{i\ge
      k}(N_i(\pi_{i,j}^*\omega)+O(m^{-k}))\\
    &=\lim_{k\to\infty}(N_j(\omega)+O(m^{-k}))\\
    &=N_j(\omega).
  \end{split}
\end{equation}
\par We now  prove~(\ref{eq:limit_curr_wadsys_s7}); fix $U\Subset X_j$
and $\varepsilon>0$; find
$\{\omega_p\}_{p=1}^{S(\varepsilon)}\subset\Omega$ such that the sets
$\{\spt\omega_p\}_{p=1}^{S(\varepsilon)}$ are pairwise disjoint,
$\spt\omega_p\cap X_j\subset U$ and
\begin{equation}
  \label{eq:limit_curr_wadsys_p43}
  \cmass N_j.(U)\le\sum_{p=1}^{S(\varepsilon)}N_j(\omega_p)+\varepsilon.
\end{equation}
As the sets $\{\spt\pi_{\infty,j}^*\omega_p\cap
X_\infty\}_{p=1}^{S(\varepsilon)}$ are pairwise disjoint and:
\begin{equation}
  \label{eq:limit_curr_wadsys_p44}
  \spt\pi_{\infty,j}^*\omega_p\cap
  X_\infty\subset\pi_{\infty,j}^{-1}(U),
\end{equation}
we have
\begin{equation}
  \label{eq:limit_curr_wadsys_p45}
  \begin{split}
    \cmass
    N_j.(U)&\le\sum_{p=1}^{S(\varepsilon)}N_\infty(\pi_{\infty,j}^*\omega_p)+\varepsilon\\
    &\le\cmass N_\infty.(\pi_{\infty,j}^{-1}(U))+\varepsilon,
  \end{split}
\end{equation}
which establishes 
\begin{equation}
  \label{eq:limit_curr_wadsys_p46}
  \cmass N_j.(U)\le \cmass N_\infty.(\pi_{\infty,j}^{-1}(U)).
\end{equation}
We now find a finite subcomplex $L$ of $X_j^{(k)}$ (where $k$ depends
on $\varepsilon$) with $U\subset L$ and:
\begin{equation}
  \label{eq:limit_curr_wadsys_p47}
  \cmass N_j.(L\setminus U)\le\varepsilon.
\end{equation}
For each gallery-connected component $L_a$ of $L$, the argument of
\textbf{Step 7} applied to the inverse system
$\{\pi_{i,j}^{-1}(L_a)\}_{i\in I}$ shows that:
\begin{equation}
  \label{eq:limit_curr_wadsys_p48}
  N_i\on\pi_{i,j}^{-1}(L)\to N_\infty\on\pi_{\infty,j}^{-1}(L).
\end{equation}
By lower semicontinuity of the mass:
\begin{equation}
  \label{eq:limit_curr_wadsys_p49}
  \begin{split}
    \cmass N_\infty.(\pi_{\infty,j}^{-1}(U))&\le\cmass
    N_\infty\on\pi_{\infty,j}^{-1}(L).(Z)\\
    &\le\liminf_{i\to\infty}\cmass
    N_i\on\pi_{i,j}^{-1}(L).(Z)\\
    &=\cmass N_j.(L)\\
    &\le \cmass N_j.(U)+\varepsilon.
  \end{split}
\end{equation}
Thus (\ref{eq:limit_curr_wadsys_p46}),
(\ref{eq:limit_curr_wadsys_p49}) give 
\begin{equation}
  \label{eq:limit_curr_wadsys_p50}
  \cmass N_\infty.(\pi_{\infty,j}^{-1}(U))=\cmass N_j.(U),
\end{equation}
and a similar argument establishes:
\begin{equation}
  \label{eq:limit_curr_wadsys_p51}
  \cmass \partial N_\infty.(\pi_{\infty,j}^{-1}(U))=\cmass \partial N_j.(U).
\end{equation}
\par We now prove~(\ref{eq:limit_curr_wadsys_s6}). Note that
by~(\ref{eq:limit_curr_wadsys_s7}), as $\mu_j\to\mu_\infty$ we have:
\begin{equation}
  \label{eq:limit_curr_wadsys_p52}
  \cmass N_\infty.\le\mu_\infty,
\end{equation}
and thus \cite[Thm.~\sync thm:repr.]{curr_alb} gives:
\begin{equation}
  \label{eq:limit_curr_wadsys_p53}
  N_\infty = \lambda D_{\infty,1}\wedge\cdots\wedge
  D_{\infty,N}\,\mu_\infty;
\end{equation}
concretely, $\lambda$ is, up to a sign, the derivative $\frac{d\cmass
  N_\infty.}{d\mu_\infty}$ and so:
\begin{equation}
  \label{eq:limit_curr_wadsys_p54}
  |\lambda|\le 1.
\end{equation}
Fix $i\in I$ and $\sigma\in\ccell X_i,.$; then
by~(\ref{eq:limit_curr_wadsys_p53}) we have:
\begin{equation}
  \label{eq:limit_curr_wadsys_p55}
  (N_\infty\on\pi^{-1}_{\infty,i}(\sigma))\on dx_2\wedge\cdots\wedge
  dx_N = \lambda \chi_{\pi_{\infty,i}^{-1}(\sigma)}D_{\infty,1}\,\mu_\infty.
\end{equation}
Using~(\ref{eq:limit_curr_wadsys_s5}) we have:
\begin{equation}
  \label{eq:limit_curr_wadsys_p56}
  \pi_{\infty,i\#}(N_\infty\on\pi^{-1}_{\infty,i}(\sigma)\on dx_2\wedge\cdots\wedge dx_N )=(N_i\on\sigma)\on
  dx_2\wedge\cdots\wedge dx_N,
\end{equation}
which combined with~(\ref{eq:limit_curr_wadsys_p53}) yields:
\begin{equation}
  \label{eq:limit_curr_wadsys_p57}
  \pi_{\infty,i\#}(\lambda
  \chi_{\pi_{\infty,i}^{-1}(\sigma)}D_{\infty,1}\,\mu_\infty) = \chi_{\sigma}D_{i,1}\,\mu_i;
\end{equation}
using the disintegration theorem, (\ref{eq:limit_curr_wadsys_s3}) and
(\ref{eq:limit_curr_wadsys_p54}) we conclude that $\lambda=1$
$\mu_\infty$-a.e.~on $\pi^{-1}_{\infty,i}(\sigma)$.
\end{proof}
\section{The Examples}
\label{sec:exas}
In this Section we discuss how to use
Theorem~\ref{thm:limit_curr_wadsys} to obtain $N$-dimensional normal
currrents whose supports are purely $2$-unrectifiable. The idea is
inspired by the repeated use of branched covers, akin to what happens
for Pontryagin surfaces (see \cite{buyalo_nagata}). However, these
examples are topologically distinct from Pontryagin surfaces, e.g.~the
cohomology groups are different. Before proving the $2$-unrectifiability
we show that the $2$-dimensional versions of these examples do not
contain embedded surfaces  (Thm.~\ref{thm:no_emb}), in sharp contrast
with the case of Carnot groups which contain H\"older surfaces.
\def\oldec{_\text{\normalfont old}}
\def\nwdec{_\text{\normalfont new}}
\def\odec{_\text{\normalfont o}}
\def\cdec{_\text{\normalfont c}}
\def\adec{_\text{\normalfont a}}
\begin{constr}[$2$-Branched covers of ${[0,1]^N}$]
  \label{constr:branched_cvs}
  Fix $1\le\alpha<\beta\le N$ and let
  $q_{\alpha,\beta}:[0,1]^N\to[0,1]^2$ denote the projection on the
  $(\alpha,\beta)$-plane. We subdivide $[0,1]^N$ into $5^N$ equal
  cells to get a cube-complex $\sigma\oldec$ and we subdivide
  $[0,1]^2$ into $5^2$ equal cells to get a square-complex
  $\Sigma\oldec$. Subdivide $\Sigma\oldec$ into the following
  square-complexes which have disjoint interiors:
  \begin{enumerate}
  \item $\Sigma\cdec$ corresponding to the central square; let $\sigma\cdec=q_{\alpha,\beta}^{-1}(\Sigma\cdec)$;
  \item $\Sigma\adec$ corresponding to the middle ring of squares; let $\sigma\adec=q_{\alpha,\beta}^{-1}(\Sigma\adec)$;
  \item $\Sigma\odec$ corresponding to the outer ring of squares; let $\sigma\odec=q_{\alpha,\beta}^{-1}(\Sigma\odec)$.
  \end{enumerate}
  As $\pi_1(\sigma\adec)=\zahlen$ we let $\tilde{\sigma\adec}$ be the
  cube-complex which is the double cover of $\sigma\adec$, and let
  $\sigma\nwdec$ be the cube-complex obtained by gluing
  $\tilde{\sigma\adec}$, $\sigma\cdec$ and $\sigma\odec$ so that any
  two points on the boundary of $\tilde{\sigma\adec}$ which map to the
  same point of $\sigma\adec$ are identified. The covering map:
  \begin{equation}
    \label{eq:branched_cvs_1}
    \tilde{\sigma\adec} \to \sigma\adec
  \end{equation}
  induces an open, cellular, surjective map:
  \begin{equation}
    \label{eq:branched_cvs_2}
    \phi_{\alpha,\beta}:\sigma\nwdec\to\sigma\oldec
  \end{equation}
  which restricts to an isometry on each face of $\sigma\nwdec$.
  \par Whenever $\sigma\oldec\simeq[0,5^{-i}]^N$ one can similarly
  apply a rescaled version of the previous construction to obtain
  $\sigma\nwdec$.
  \par Given a measure $\mu\oldec$ on $\sigma\oldec$ which is a
  multiple of Lebesgue measure,  one can evenly
  split $\mu\oldec$ across the $N$-cells of $\sigma\nwdec$ that
  project to the same cell of $\sigma\oldec$ to obtain $\mu\nwdec$
  which satisfies $\phi_{\alpha,\beta\#}\mu\nwdec=\mu\oldec$
\end{constr}
\begin{constr}[The examples]
  \label{constr:exas}
  Let $X_0$ be an orientable cube-complex which is the union of its
  $N$-cells which are all isometric to $[0,1]^N$. We will also assume
  that there is a doubling measure $\mu_0$ on $X_0$ which restricts to
  a multiple of Lebesgue measure on each element of $\ccell X_0,.$ and
  such that $(X_0,\mu_0)$ is a $(1,1)$-PI space.
  \par Let $I=\{0\}\cup\natural$ and for $i\ge 1$ choose
  $1\le\alpha_i<\beta_i\le N$ such that each possible pair
  $(\alpha,\beta)$ in the $2$-ordered combinations of $\{1,\cdots,N\}$
  occurss i.o.
  \par We obtain an inverse system if we construct $(X_{i+1},\mu_{i+1})$
  by applying the branched covering
  construction~\ref{constr:branched_cvs} to each $\sigma\in\ccell
  X_i,.$ using the plane $(\alpha_i,\beta_i)$.
\end{constr}
\begin{thm}[The examples are AIS and WAIS]
  \label{thm:exas_are_invs}
  The system $\{(X_i,\mu_i)\}$ is both AIS and WAIS; moreover each
  weak tangent $Y$ of the inverse limit is, up to a dilating
  factor in $[1/5,1)$, the inverse limit of a system  obtained using
  Construction~\ref{constr:exas}.
\end{thm}
\begin{proof}
  \par\noindent\texttt{Step 1: Admissibility.}
  \par The only assertion that requires justification is
  \textbf{(IBGeom)}, as the other axioms follow from the recursive
  construction. Fix $x_i\in X_i$ and for $j\le i$ let
  $x_j=\pi_{i,j}(x_i)$. In passing from $x_j$ to $x_{j+1}$ the
  cardinality of the link can increase only under the following
  circumstance: if $x_j\in\sigma$, where $\sigma\in\ccell X_j,.$, and
  if
  \begin{equation}
    \label{eq:exas_are_invs_p1}
    q_{\alpha_j,\beta_j}:\sigma\to\real^2
  \end{equation}
  denotes the projection on the $(\alpha_j,\beta_j)$-plane,
  one must have
  \begin{equation}
    \label{eq:exas_are_invs_p2}
    q_{\alpha_j,\beta_j}(x_j)\in\partial\Sigma\adec,
  \end{equation}
  using the notation of Construction~\ref{constr:branched_cvs}.
  Note that for cubes $\sigma$, $\sigma'$ of different generations $k$, $k'$, if
  $(\alpha_k,\beta_k)=(\alpha_{k'},\beta_{k'})$ the corresponding 
  $\partial\Sigma\adec$ and $\partial\Sigma'\adec$ are disjoint, and so
  if~(\ref{eq:exas_are_invs_p2}) occurrs for a given pair
  $(\alpha_j,\beta_j)$, it can occurr only once for that pair. In that
  case, the cardinality of the link can at most double; as there are at
  most $N\choose 2$ distinct $(\alpha,\beta)$-pairs, \textbf{(IBGeom)}
  follows.
  \par\noindent\texttt{Step 2: Weak tangents}
\par  Assuming $(\lambda_n X_\infty,p_n,\mu_n)\to (Y_\infty, q,\nu_\infty)$, up to rescaling
  the metric on $Y_\infty$ by a dilating factor in $[1/5,1)$, we can
  assume that $\lambda_n=5^{m_n}$. Choose compatible systems of
  basepoints $\{p_{n,k}\}$ so that
  \begin{equation}
    \label{eq:exas_are_invs_p3} \pi_{\infty,k}(p_n)=p_{n,k}.
  \end{equation}
  A compactness argument shows that, for each $k\in I$
  $(\lambda_nX_{k+m_n},p_{n,k+m_n},\mu_{n,k+m_n})$ subconverges to some $(Y_k,q_k,\nu_k)$,
  and the $\{(Y_k,\nu_k)\}_{k\in I}$ form an AIS/WAIS whose inverse limit is $(Y_\infty,\nu_\infty)$.
\end{proof}
In the following it will be useful to consider the case of
Construction~\ref{constr:exas} where $X_0$ is $2$-dimensional, and
where the branching is applied for $\infty$-many values of $i$, but
not necessarily all values of $i$; in this case we will use
$\{Y_i\}_{i\in I}$ (omitting the measures)
to denote the inverse system and $Y_\infty$ the inverse limit.
\begin{thm}
  \label{thm:no_emb}
  There is no embedding:
  \begin{equation}
    \label{eq:no_emb_s1}
    j:\bar D^2 = [0,1]^2 \to Y_\infty,
  \end{equation}
  where $\bar D^2$ is the $2$-dimensional closed disk.
\end{thm}
\begin{proof}
  We will argue by contradiction.
  \par\noindent\texttt{Step 1: $j(\bar D^2)$ has nonempty interior.}
  \par Assume that for each $k\in\natural$ $\pi_{\infty, k}(j(\bar D^2))$
  does not contain an $2$-cell of $Y_k$; then $\pi_{\infty,k}(j(\bar
  D^2))$ can be retracted to the $1$-skeleton of $Y_k$, and
  by~(\ref{eq:adsys_pi_p1}) the retraction can be chosen to produce a
  $1$-dimensional simplicial complex which is an
  $O(m^{-k})$-approximation of $j(\bar D^2)$. Letting $k\to\infty$ we
  conclude that $j(\bar D^2)$ has topological dimension $1$, which
  contradicts that $\bar D^2$ has topological dimension $2$. 
  \par Fix $k_0$ and $\sigma$ such that $\sigma\in\ccell Y_{k_0},2.$,
  $\sigma\subset\pi_{\infty,k_0}(j(\bar D^2))$. As in
  Construction~\ref{constr:branched_cvs} we let 
  \begin{equation}
    \label{eq:no_emb_p1}
    \tilde{\sigma}=\tilde{\sigma\cdec}\cup\tilde{\sigma\adec}\cup\tilde{\sigma\odec}=\pi_{k_0}^{-1}(\sigma);
  \end{equation}
  as $\pi_{k_0}|\tilde{\sigma\cdec}\cup\tilde{\sigma\odec}$ is injective
  we deduce
  \begin{equation}
    \label{eq:no_emb_p2}
    \tilde{\sigma\cdec}\cup\tilde{\sigma\odec}\subset\pi_{\infty,k_0+1}(j(\bar D^2)).
  \end{equation}
  Were
  \begin{equation}
    \label{eq:no_emb_p3}
    \tilde{\sigma\adec}\setminus\pi_{\infty,k_0+1}(j(\bar D^2))\ne\emptyset,
  \end{equation}
  we could retract via a map $r$ the set:
  \begin{equation}
    \label{eq:no_emb_p4}
    \tilde{\sigma\adec}\cap\pi_{\infty,k_0+1}(j(\bar D^2))
  \end{equation}
  to $\partial\tilde{\sigma\adec}$ while keeping
  $Y_{k_0+1}\setminus\tilde\sigma\adec$ fixed, and then
  $(r\circ\pi_{\infty,k_0+1})^{-1}(\tilde{\sigma\cdec})$ and
  $(r\circ\pi_{\infty,k_0+1})^{-1}(Y_{k_0+1}\setminus\tilde{\sigma\cdec})$
  would disconnect $j(\bar D^2)$, a contradiction. Hence:
  \begin{equation}
    \label{eq:no_emb_p5}
    \tilde{\sigma}\subset\pi_{\infty,k_0+1}(j(\bar D^2)).
  \end{equation}
  \par We can iterate the previous argument on each cell of
  $\tilde{\sigma}^{(1)}$ and continue to conclude that
  \begin{equation}
    \label{eq:no_emb_p6}
    \pi_{\infty,k_0}^{-1}(\sigma)\subset j(\bar D^2),
  \end{equation}
  showing that $j(\bar D^2)$ has nonempty interior.
  \par\noindent\texttt{Step 2: Simple connectedness.}
  \par As any point of $\bar D^2$ has a simply-connected neighbourhood,
  it suffices to show that each non-empty open $U\subset Y_\infty$ is not
  simply-connected.  In fact, for some $k$, there is a loop $\gamma$ in
  $U$ such that $\pi_{\infty,k}\circ\gamma$ is a generator of the
  fundamental group of $\partial\sigma$, where $\sigma\in\ccell
  X_k,2.$. But then $\pi_{\infty,k+1}\circ\gamma$ is not homotopically
  trivial in $Y_{k+1}$ by construction of the branched cover.
\end{proof}
\begin{thm}
  \label{thm:unrect}
  $X_\infty$ is purely $2$-unrectifiable.
\end{thm}
\begin{proof}
  \par\noindent\texttt{Step 1: Reduction to $2$ dimensions.}
  \par We argue by contradiction assuming that there are a compact set
  $K$ and a Lipschitz map $f$:
  \begin{equation}
    \label{eq:unrect_p1}
    f:K\subset\real^2\to X_\infty
  \end{equation}
  with $\hmeas 2.(f(K))>0$. Without loss of generality we can assume
  that $\inf I=0$ and $X_0=[0,1]^N$. 
  \par As $k$ varies in $I$ consider $1$-Lipschitz functions 
  \begin{equation}
    \label{eq:unrect_p2}
    \psi:X_k\to\real:
  \end{equation}
  the collection $\{\pi^*_{\infty,k}\psi\}_{k,\psi}$ uniformly separates
  points in $X_\infty$; concretely, given $x,x'\in X_\infty$ one can
  find $k\in I$ and a $1$-Lispchitz 
  \begin{equation}
    \label{eq:unrect_p3}
    \psi:X_k\to\real
  \end{equation}
  such that:
  \begin{equation}
    \label{eq:unrect_p4}
    \frac{1}{2}d_{X_\infty}(x,x')\le\psi(\pi_{\infty,k}(x))-\psi(\pi_{\infty,k}(x'));
  \end{equation}
  in fact, we know from~(\ref{eq:pi_uf_assmp_1}) that:
  \begin{equation}
    \label{eq:unrect_p5}
    \left |
      d_{X_\infty}(x,x') - (\pi^*_{\infty,k}d_{X_k})(x,x')
    \right|=O(m^{-k}).
  \end{equation}
  \par By the   Stone-Weierstrass Theorem for Lipschitz
  algebras~\cite[Thm.~4.1.8]{weaver_book99} the unital algebra generated by
  $\{\pi^*_{\infty,k}\psi\}_{k,\psi}$ is weak* dense in $\lipalg
  X_\infty.$. Moreover, each cell of $X_k$ is isometric to a cube in
  $\real^N$ and so the differential $d\pi^{*}_{\infty,k}\psi$ is a
  linear combination of $dx_1,\cdots,dx_N$ (we abuse notation and just
  write $x_\alpha$ for $x_\alpha\circ\pi_{\infty,0}$). Thus, any derivation
  $D\in\wder\nu.$, $\nu$ being a Radon measure on $X_\infty$, is
  determined by $\{Dx_\alpha\}_{\alpha=1,\cdots,N}$.
  \par Let $\partial_1,\partial_2$ be the standard basis of
  $\wder{\lebmeas 2.}\on K.$; as $\hmeas 2.(f(K))>0$ the
  area formula \cite[Thm.~8.2]{ambrosio-rectifiability} implies that the derivations
  $f_{\#}\partial_1, f_{\#}\partial_2$ must be independent, and so there
  must  be $\alpha<\beta$ and a  compact $\tilde K\subset K$ with
  $\lebmeas 2.(\tilde K)>0$ such that
  \begin{equation}
    \label{eq:unrect_p6}
    \det
    \begin{pmatrix}
      \partial_1(f\circ x_\alpha) & \partial_1(f\circ x_\beta)\\
      \partial_2(f\circ x_\alpha) & \partial_2(f\circ x_\beta)
    \end{pmatrix}
  \end{equation}
  is uniformly bounded away from $0$ on $\tilde K$.
  \par For each $k\in I$ let $Z_k$ be the cube-complex obtained by
  collapsing each cell of $X_k$ in the directions
  $\{x_\gamma\}_{\gamma\ne\alpha,\beta}$. The $\{Z_k\}_{k\in I}$ form
  a $2$-dimensional
  AIS/WAIS as the $\{Y_k\}$ discussed before Theorem~\ref{thm:no_emb}. In
  particular, if $q_{\alpha\beta}$ denotes projection on the
  $(\alpha,\beta)$ plane, we have commutative diagrams:
  \begin{equation}
    \label{eq:unrect_p7}
    \xy
    (0,20)*+{X_\infty}="xf"; (20,20)*+{Z_\infty}="zf";
    (0,10)*+{X_i}="xi"; (20,10)*+{Z_i}="zi";
    (0,0)*+{[0,1]^N}="cn"; (20,0)*+{[0,1]^2}="q";
    {\ar "xf"; "zf"}?*!/_2mm/{\psi_\infty};
    {\ar "xi"; "zi"}?*!/_2mm/{\psi_i};
    {\ar "cn"; "q"}?*!/_2mm/{\psi_0};
    {\ar "xf"; "xi"}?*!/^4mm/{\pi_{\infty,i}};
    {\ar "xi"; "cn"}?*!/^4mm/{\pi_{i,0}};
    {\ar "zf"; "zi"}?*!/_4mm/{\tilde\pi_{\infty,i}};
    {\ar "zi"; "q"}?*!/_10mm/{\tilde\pi_{i,0}=q_{\alpha\beta}};
    \endxy
  \end{equation}
  where the functions $\{\psi_i\}$ are $1$-Lipschitz.
  \par\noindent\texttt{Step 2: Blow-up.}
  \par Consider:
  \begin{equation}
    \label{eq:unrect_p8}
    \xy
    (0,20)*+{\tilde K}="source";
    (20,20)*+{Z_\infty}="zf";
    (30,10)*+{;};
    (20,0)*+{[0,1]^2}="q";
    {\ar "source"; "zf"}?*!/_2mm/{\psi_\infty\circ f};
    {\ar "zf"; "q"}?*!/_4mm/{q_{\alpha\beta}};
    {\ar@{->}_{q_{\alpha\beta}\circ\psi_\infty\circ f} "source"; "q"};
    \endxy 
  \end{equation}
  blowing up $\psi_\infty\circ f$ at a density point of $\tilde K$ we
  obtain a diagram:
  \begin{equation}
    \label{eq:unrect_p9}
    \xy
    (0,20)*+{\real^2}="source";
    (20,20)*+{Y_\infty}="zf";
    (30,10)*+{,};
    (20,0)*+{\real^2}="q";
    {\ar "source"; "zf"}?*!/_2mm/{G};
    {\ar "zf"; "q"}?*!/_4mm/{q_{\alpha\beta}};
    {\ar@{->}_{q_{\alpha\beta}\circ G} "source"; "q"};
    \endxy 
  \end{equation}
  where $q_{\alpha\beta}\circ G$ must be linear; by the bounds
  on~(\ref{eq:unrect_p6}) $q_{\alpha\beta}\circ G$ is non-singular, so
  taking $\bar D^2\subset\real^2$
  \begin{equation}
    \label{eq:unrect_p10}
    G:\bar D^2\to Y_\infty
  \end{equation}
  gives an embedding contradicting Theorem~\ref{thm:no_emb} (one can
  also argue that $q_{\alpha\beta}$ is not injective).
\end{proof}
\section{Approximation by cubical currents}
\label{sec:appx_cube}
In this Section we discuss how to use currents associated to cube
complexes to approximate normal currents (sitting in $l^\infty$) in
the flat and weak topologies.
\begin{defn}
  \label{defn:cubical_current}
  Let $X$ be an $N$-dimensional finite cube-complex. A \textbf{cubical
    current} (supported on $X$) is a normal current of the form:
  \begin{equation}
    \label{eq:cubical_current_1}
    \sum_{\sigma\in\ccell X,.}\weight(\sigma)\mkcurr \sigma..
  \end{equation}
\end{defn}
\def\pjdec{_\text{\normalfont pj}}
\def\smdec{_\text{\normalfont sm}}
\def\cbdec{_\text{\normalfont cb}}
\def\cpbdec{_\text{$\partial$\normalfont cb}}
\def\fldec{_\text{\normalfont {f}{l}}} 
\def\hypplane{\mathscr{S}}
\def\dfodec{_\text{\normalfont df,$1$}}
\def\dftdec{_\text{\normalfont df,$2$}}
\begin{thm}
  \label{thm:cubical_appx}
  Let $T$ be a compactly-supported normal current in $l^\infty$. For
  each $\varepsilon>0$ there is a compactly-supported cubical current
  $T\cdec$ such that:
  \begin{align}
    \label{eq:cubical_appx_s1}
    \flatnrm T-T\cdec. &\le\varepsilon\\
    \label{eq:cubical_appx_s2}
    \ncnrm T\cdec.&\le\ncnrm T.+\varepsilon\\
    \label{eq:cubical_appx_s3}
    \spt T\cdec&\subset B(\spt T,\varepsilon).
  \end{align}
  In particular, $T$ can be approximated in the weak topology by a
  sequence of cubical currents $\{T_n\}$ with:
  \begin{align}
    \label{eq:cubical_appx_s4}
    \lim_{n\to\infty}\ncnrm T_n.&=\ncnrm T.\\
    \label{eq:cubical_appx_s5}
    \lim_{n\to\infty}\sup_{y\in\spt T_n}\inf_{z\in\spt T}\|y-z\|_\infty&=0.
  \end{align}
\end{thm}
\begin{rem}
  \label{rem:bounded_app}
  In Theorem~\ref{thm:cubical_appx} one can replace $l^\infty$ with a
  Banach space having the \emph{bounded approximation property}; the
  only difference is the bound~(\ref{eq:cubical_appx_s3}) which
  worsens to:
  \begin{equation}
    \label{eq:bounded_app_1}
    \ncnrm T\cdec.\le C\ncnrm T.+\varepsilon,
  \end{equation}
  $C$ being the norm of the projections used in the bounded
  approximation property.
\end{rem}
\begin{proof}[Proof of Theorem~\ref{thm:cubical_appx}]
  \par\noindent\texttt{Step 1: Choice of a projection.}
  \par The space $l^\infty$ has the bounded approximation property where
  projections can be taken to have norm $1$ (see \cite[Lem.~5.7]{paolini_acyclic}). Concretely, for each parameter $\varepsilon\pjdec>0$
  we can choose a norm-$1$ projection
  \begin{equation}
    \label{eq:cubical_appx_p1}
    \pi: l^\infty\to\hypplane,
  \end{equation}
  $\hypplane$ begin a finite-dimensional hyperplane in $l^\infty$, such
  that:
  \begin{align}
    \label{eq:cubical_appx_p2}
    \pi(\spt T)&\subset B(\spt T,\varepsilon\pjdec)\\
    \label{eq:cubical_appx_p3}
    \sup_{y\in\spt T}\|\pi(y)-y\|_{\infty}&\le\varepsilon\pjdec.
  \end{align}
  Let
  \begin{equation}
    \label{eq:cubical_appx_p4}
    \begin{aligned}
      H:[0,1]\times\spt T &\to l^\infty\\
      (t,x)&\mapsto (1-t)x+t\pi(x);
    \end{aligned}
  \end{equation}
  by the \emph{Homotopy formula} (see \cite[4.1.9, 4.1.10]{federer_gmt}:
  the extension to normal metric currents is straightforward)
  \begin{equation}
    \label{eq:cubical_appx_p5}
    \partial H_{\#}([0,1]\times T) = H(1,\cdot)_{\#}T - H(0,\cdot)_{\#}T
    - H_{\#}([0,1]\times\partial T),
  \end{equation}
  where if 
  \begin{equation}
    \label{eq:cubical_appx_p6}
    T = \sum_{\alpha}D_{\alpha_1}\wedge\cdots\wedge D_{\alpha_N}\,\cmass T.,
  \end{equation}
  we let:
  \begin{equation}
    \label{eq:cubical_appx_p7}
    [0,1]\times T = \sum_{\alpha}\partial_t\wedge
    D_{\alpha_1}\wedge\cdots\wedge D_{\alpha_N}\,\lebmeas .\on[0,1]\times\cmass T..
\end{equation}
We let $T\pjdec=H(1,\cdot)_{\#}T$, and note that by~(\ref{eq:cubical_appx_p3}) and~(\ref{eq:cubical_appx_p5})
we have:
\begin{equation}
  \label{eq:cubical_appx_p8}
  \flatnrm T\pjdec - T.\le\varepsilon\pjdec\ncnrm T.;
\end{equation}
as $H(1,\cdot)=\pi$ we conclude that:
\begin{align}
  \label{eq:cubical_appx_p9}
  \ncnrm T\pjdec.&\le \ncnrm T.\\
  \label{eq:cubical_appx_p10}
  \spt T\pjdec&\subset B(\spt T,\varepsilon\pjdec)\cap\hypplane.
\end{align}
\par\noindent\texttt{Step 2: Smoothing $T\pjdec$.}
\par $T\pjdec$ can be regarded as a normal current on the
finite-dimensional vector space $\hypplane$; by a standard smoothing
argument (\cite[4.1.18]{federer_gmt}) for each
$\varepsilon\smdec>0$ we can find $T\smdec$ with support in $\hypplane$
such that:
\begin{align}
  \label{eq:cubical_appx_p11}
  \flatnrm T\smdec - T\pjdec.&\le\varepsilon\smdec\\
  \label{eq:cubical_appx_p12}
  \spt T\smdec&\subset B(T\pjdec,\varepsilon\smdec)\\
  \label{eq:cubical_appx_p13}
  \ncnrm T\smdec.&\le \ncnrm T\pjdec.\\
  \label{eq:cubical_appx_p14}
  \cmass T\smdec.&\ll\lebmeas\dim\hypplane.\\
  \label{eq:cubical_appx_p15}
  \cmass\partial T\smdec.&\ll\lebmeas\dim\hypplane.,
\end{align}
where we use $\lebmeas\dim\hypplane.$ to denote Lebesgue measure on
$\hypplane$.
\par Let $\zahlen^{\dim\hypplane}$ denote the integer lattice in
$\hypplane$; because of equations~(\ref{eq:cubical_appx_p14}),
(\ref{eq:cubical_appx_p15}), which express absolute continuity of
$\cmass T\smdec.$ and $\cmass \partial T\smdec.$ wrt.~the Lebesgue
measure, we can approximate these currents in the \emph{flat norm} by cubical
currents: i.e.~for each $\varepsilon\cbdec>0$ we can find 
cubical currents $T\cbdec$, $T\cpbdec$ and $n_0\in\natural$ with
$2^{-n_0}\le\varepsilon\cbdec$ and where:
\begin{description}
\item[(Sm1)] The supports of $T\cbdec$ and $T\cpbdec$ belong,
  respectively, to
  the $N$-skeleton and the $(N-1)$-skeleton of the cubulation of $\hypplane$ associated to
  $2^{-n_0}\zahlen^{\dim\hypplane}$;
\item[(Sm2)] $T\cbdec$ is $N$-dimensional and $T\cpbdec$ is
  $(N-1)$-dimensional and one has:
  \begin{align}
    \label{eq:cubical_appx_p16}
    \mcnrm T\cbdec.&\le\mcnrm T\smdec.\\
    \label{eq:cubical_appx_p17}
    \mcnrm T\cpbdec.&\le\mcnrm\partial T\smdec.\\
    \label{eq:cubical_appx_p18}
    \flatnrm T\cbdec - T\smdec. + \flatnrm T\cpbdec - \partial
    T\smdec.&\le\varepsilon\cbdec\\
\label{eq:cubical_appx_p18bis}
\spt T\cbdec\cup\spt T\cpbdec&\subset B(\spt T\smdec,\varepsilon\cbdec).
  \end{align}
\par\noindent\texttt{Step 3: Application of the Deformation Theorem in
  $\hypplane$.}
\par The argument is then completed applying the Federer-Fleming
Deformation Theorem along the lines of \cite[4.2.23,
4.2.24]{federer_gmt}. 
Concretely, (\ref{eq:cubical_appx_p18}) gives:
\begin{equation}
  \label{eq:cubical_appx_p19}
  \flatnrm \partial T\cbdec - T\cpbdec.\le 2\varepsilon\cbdec,
\end{equation}
where the flat distance can be actually computed in the metric space $B(\spt T\smdec,
3\varepsilon\cbdec)\cap\hypplane$, 
and thus there is a normal current $T\fldec$ supported in $\hypplane$
with:
\begin{align}
  \label{eq:cubical_appx_p20}
  \mcnrm \partial T\cbdec - T\cpbdec - \partial T\fldec. + \mcnrm
  T\fldec. &\le 3\varepsilon\cbdec\\
  \label{eq:cubical_appx_p21}
  \spt T\fldec &\subset B(\spt T\smdec, 3\varepsilon\cbdec).
\end{align}
\end{description}
\par For any $\varepsilon\dfodec>0$, applying the Deformation Theorem
we can find $n_1\in\natural$, an $(N-1)$-cubical current $T\dfodec$
and an $N$-normal current $S\dfodec$ such that:
\begin{description}
\item[(DF1)] The support of $T\dfodec$ belongs to the $(N-1)$-skeleton
  of the cubulation of $\hypplane$ with vertices on
  $2^{-n_1}\zahlen^{\dim\hypplane}$, where
  $2^{-n_1}\le\varepsilon\dfodec$;
\item[(DF2)] Letting $C_\hypplane$ denote a constant depending only on
  $\hypplane$, one has:
  \begin{align}
    \label{eq:cubical_appx_p22}
    \partial T\cbdec - T\cpbdec - \partial T\fldec &= T\dfodec + \partial S\dfodec\\
    \label{eq:cubical_appx_p23}
    \mcnrm S\dfodec.&\le C_\hypplane\varepsilon\dfodec\mcnrm \partial
    T\cbdec - T\cpbdec - \partial T\fldec.\\
    \label{eq:cubical_appx_p24}
    \mcnrm T\dfodec.&\le C_\hypplane\left[
      \mcnrm \partial
      T\cbdec - T\cpbdec - \partial T\fldec. +
      \varepsilon\dfodec\mcnrm\partial T\cpbdec.
    \right]\\
    \label{eq:cubical_appx_p25}
    \spt T\dfodec \cup\spt S\dfodec &\subset B(\spt T\smdec,
    3\varepsilon\cbdec +3\varepsilon\dfodec).
  \end{align}
\end{description}
Note that as $C_\hypplane$ depends only on $\hypplane$, the
rhs.~of~(\ref{eq:cubical_appx_p23}) can be made arbitrarly small by
choosing $\varepsilon\dfodec$ sufficiently small. Note also that in
applying the Deformation Theorem we used that the boundary of
$T\cpbdec$ is cubical in order to claim that $T\dfodec$ is cubical.
\par As by (\ref{eq:cubical_appx_p22}) the boundary of
$S\dfodec+T\fldec$ is cubical, we can reiterate and for
$\varepsilon\dftdec$ we can find $n_2\in\natural$, an $N$-cubical
current $T\dftdec$ and an $(N+1)$-normal current $S\dftdec$ such that:
\begin{description}
\item[(DF3)] The support of $T\dftdec$ belongs to the $N$-skeleton of
  the cubulation of $\hypplane$ with vertices on
  $2^{-n_2}\zahlen^{\dim\hypplane}$, where
  $2^{-n_2}\le\varepsilon\dftdec$;
\item[(DF4)] One has:
  \begin{align}
    \label{eq:cubical_appx_p26}
    S\dfodec + T\fldec &= \partial S\dftdec + T\dftdec\\
    \label{eq:cubical_appx_p27}
    \mcnrm T\dftdec. &\le C_\hypplane \left[
      \mcnrm S\dfodec + T\fldec. + \varepsilon\dftdec\mcnrm
      \partial S\dfodec + \partial T\fldec.
    \right]\\
    \label{eq:cubical_appx_p28}
    \mcnrm S\dftdec. &\le C_\hypplane\varepsilon\dftdec\mcnrm S\dfodec+T\fldec.\\
    \label{eq:cubical_appx_p29}
    \spt T\dftdec \cup \spt S\dftdec &\subset B(\spt T\smdec,
    4\varepsilon\cbdec +4\varepsilon\dfodec +2\varepsilon\dftdec).
  \end{align}
\end{description}
The desired current $T\cdec$ is then obtained by letting
\begin{equation}
  \label{eq:cubical_appx_p30}
  T\cdec = T\cbdec - T\dftdec
\end{equation}
and by choosing the $\varepsilon$-parameters sufficiently small.
\end{proof}
\bibliographystyle{alpha}
\bibliography{sponge_biblio}
\end{document}